%% file: main.tex
\documentclass[11pt]{article}

\usepackage[a4paper, total={6in, 8in}]{geometry}
\input{preamble.tex}

\usepackage{hyperref}

\usepackage{graphicx}

\title{Scaling Laws for Majority-based Opinion Dynamics in the Presence of Stubborn Agents}
\author{Luke Meredith and Arpan Mukhopadhyay\\
University of Warwick}

\date{}

\begin{document}
\maketitle

\begin{abstract}
In a multi-agent system, there are often stubborn followers of specific opinions or beliefs. Motivated by this observation, in this paper, we aim to understand how stubborn agents affect the distribution of opinions in a network where both stubborn and non-stubborn agents interact with each other. To do so, we assume that all agents have an opinion in the set $\{0,1\}$ and  each non-stubborn agent updates its opinion according to the $2k$\textit{-choices rule}, where the agent samples $2k$ neighbours (including both stubborn and non-stubborn neighbours) uniformly at random and adopts the majority opinion among the sampled group of neighbours and itself.  We assume that a proportion of agents, $\gamma_i$, are stubborn followers of opinion $i\in \{0,1\}$. It is natural to expect that the steady-state distribution of the opinions in the network will be dominated by the opinion with the larger proportion of stubborn followers. We show that while this is true, the time to reach steady-state depends heavily on the values of the parameters $\gamma_0$ and $\gamma_1$. When the individual values of these parameters, as well as their difference, are small, it can take an exponentially long time (in the network size) to reach the steady-state. In sharp contrast, when at least one of the parameters $\gamma_0$ and $\gamma_1$ is large, the network reaches the steady-state in a time that is only logarithmic in the network size. Hence, there exists a sharp phase transition in the network dynamics based on the proportions of stubborn agents. We also characterise the behaviour of the system when the parameters $\gamma_0$ and $\gamma_1$ lie on the boundary of the phase transition. In this boundary region, we show using Stein's method that the dynamics are driven by a diffusion process which takes polynomial time to mix.

\end{abstract}


\section{Introduction}

Consensus formation in multi-agent systems is a central topic across statistical physics \cite{castellano2009statistical} and computer science \cite{byzantine_generals}. Research in this domain investigates how collective behaviour emerges from simple local interactions, and how factors like noise, bias, faults, and communication constraints shape these dynamics \cite{noisy_voter_model, pineda2009noisy}. In this paper, we analyse how \textit{stubborn agents} \cite{yildiz2011discrete} or \textit{zealots} \cite{mobilia2003does}, whose opinions remain permanently fixed, affect  consensus formation. This is motivated by both social networks, where certain individuals hold immovable beliefs or fixed roles~\cite{acemouglu2013opinion}, and distributed systems, where {\em stuck-at faults} can occur, resulting in some nodes' states being frozen permanently~\cite{finkbeiner2015detecting}.

Specifically, we consider the $2k$\textit{-choices rule}, a generalisation of the \textit{2-choices rule}, which is a well-studied distributed consensus protocol \cite{cooper_two_choices}. We assume that all agents/nodes initially have opinions in the set $\{0,1\}$. Furthermore, a fraction $\gamma_i$ of {\em stubborn} agents have a fixed opinion $i\in \{0,1\}$ at all times. The non-stubborn agents update their opinions according to independent unit-rate Poisson processes with the $2k$-choices algorithm, where an agent samples $2k$ agents from their neighbourhood uniformly at random (including both stubborn and non-stubborn agents), and then updates their own opinion to the majority opinion among the sampled agents and itself. Unlike in the classical $2k$-choices dynamics, a consensus state is no longer absorbing for these modified dynamics due to the presence of stubborn agents. It is natural to expect that, at steady-state, the opinion with the larger number of stubborn followers will hold the majority support among non-stubborn agents. We show that while this is true, the time to reach this steady-state can differ significantly, depending on the values of $\gamma_0$ and $\gamma_1$. 

In particular, we identify three distinct regimes for the parameters $(\gamma_0,\gamma_1)$. In the first regime, referred to as the {\em sub-critical regime}, the system exhibits {\em metastability} \cite{bovier2016metastability} wherein the opinion initially holding the majority support among non-stubborn agents may continue to hold its majority for an exponentially long time (in the network size $n$) before the opinion with the higher number of stubborn followers starts to dominate. In sharp contrast, in the second regime, referred to as the {\em super-critical regime}, the opinion with the higher number of stubborn followers emerges as the majority opinion among non-stubborn agents in only $O(\log n)$ time. Finally, in the {\em critical regime}, we show that the initial majority opinion among non-stubborn agents may continue to hold its majority for $\Theta(n^p)$ time where $p\in \{1/3,1/2\}$. 

One of the main challenges in our analysis comes from the fact that the dynamics of the $2k$-choices rule under stubborn agents lacks the {\em global monotonicity property} that ensures that the distance between two independent copies of the process started at two different initial states will always reduce on average. 
This makes proving mixing time bounds more challenging.
Furthermore, obtaining precise rates of convergence at the critical boundary becomes quite challenging due to the presence of degenerate roots in the drift. We overcome this challenge using Stein's method~\cite{braverman2021prelimit} which allows us to expand the generator of the underlying process and compare the drift and diffusion terms precisely. We also characterise the boundary at which the phase transition occurs. Hence, our contributions are as follows:


\begin{enumerate}
    \item We characterise the limiting steady-state distribution of the dynamics and show that it concentrates on a point where the opinion with the higher number of stubborn followers becomes the majority opinion.
    \item We identify parameter regimes where the system exhibits exponential and logarithmic convergence rates to steady-state. We also analytically characterise the phase transition boundary.
    \item At the phase transition boundary, we derive exact convergence rates that are polynomial in the network size. To derive these rates, we use Stein's framework based on generator expansion. 
\end{enumerate}

\subsection{Related Literature}

A foundational model in distributed computing is the voter model \cite{holley1975ergodic, clifford1973model} where a node simply copies the opinion of another randomly chosen node in its neighbourhood. In this model, the total number of agents holding a specific opinion is a martingale, and hence the consensus probability of an opinion is equal to its initial fraction. Due to its simplicity and duality with coalescing random walks, several variants of the voter model have been studied in the literature including models on lattices and expanders \cite{cox1989coalescing, cooper2013coalescing},  models where agents exhibit bias~\cite{arpan_JSP_20}, and models where spontaneous opinion switching/mutation occurs independent of the regular update \cite{noisy_voter_model}. 

In \cite{yildiz2011discrete}, stubborn agents were incorporated into the voter model. The authors show that, in the presence of stubborn agents with opposing opinions, convergence to consensus is not possible. Instead, if the underlying network is connected, the dynamics converge to a stationary distribution. The authors analyse the first two moments of the stationary distribution and characterise its dependence on the network structure and the positions of the stubborn individuals. Also studied is the optimal placement of agents on a graph such that stubborn agents have a maximal effect on the bias of the society toward a specific opinion.

A natural extension of the voter model is the \textit{2-choices rule}, where instead of simply copying other agents, an agent samples two other agents from their neighbourhood and adopts the majority opinion among sampled agents and itself. This rule was first analysed on complete graphs in~\cite{redner_two_choices} and on expander graphs in~\cite{cooper_two_choices,cooper2015fast}. Several variants of this rule have been considered in the recent literature. In \cite{ijcai2020p8,cruciani2021phase,mukhopadhyay2016binary} agents are assumed to exhibit a bias toward a specific opinion. A variant of this rule with competing biases was recently analysed in~\cite{capannoli2026phase}. A model where agents may occasionally fail to follow the rule was analysed in~\cite{meredith2026robustness}.
Many of these models exhibit rich phase transition properties. However, one fundamental difference between these models and our present work is that the existing models assume that all agents behave in exactly the same way, whereas in our model the agents are inherently heterogeneous due to the presence of stubborn agents with opposing opinions and non-stubborn agents. 


In \cite{arpan_JSP_20}, the authors introduce stubborn agents into the \textit{2-choices rule}.  The main focus of this work is the analysis of the mean-field limit of the system and characterisation of the stationary distribution in this limit; no result is provided for finite population sizes. Our work provides a major extension to this by fully analysing the \textit{2$k$-choices rule} with stubborn agents for finite populations. We are able to characterise the stationary distribution for all parameter configurations, identify three distinct parameter regimes, and provide scaling laws for the rate of convergence to stationarity in each regime.

\subsection{General Notations}
We use $\RR$ and $\NN$ to denote the set of real and natural numbers, respectively. For $x\in \RR$, $\floor{x}$ and $\ceil{x}$  denote the highest integer not exceeding $x$ and the smallest integer above $x$, respectively. Similarly, $[x]_+=\max(0,x)$ and $[x]_-=\max(0,-x)$. For two positive functions $f$ and $g$ defined on positive integers we write $f(n)=O(g(n))$ when $\limsup_{n\to \infty}f(n)/g(n) < \infty$; $f(n)=o(g(n))$ when $\limsup_{n\to \infty}f(n)/g(n)= 0$; $f(n)=\Omega(g(n))$ when $g(n)=O(f(n))$; $f(n)=\omega(g(n))$ when $g(n)=o(f(n))$.

\subsection{Organisation}
The remaining sections of this paper are organised as follows. Section \ref{Sec: Model 2k Asymmetric} introduces the $2k$-choices model with stubborn agents and defines some preliminary objects required for our analysis. In Section \ref{sec:Main Results}, we state and discuss our main results. In Section \ref{Sec:Meta Lemma}, we state and prove a central lemma required throughout our analysis. In Sections~\ref{proof:stationary}-\ref{proof:criticality}, we provide the detailed proofs of the main theorems. Section \ref{Sec:conclusion} concludes the paper.

\section{Model and Preliminaries}
\label{Sec: Model 2k Asymmetric}

We consider a fully connected network of $n\in\mathbb{N}$ agents/nodes where each agent has an opinion in the set $\{0,1\}$ at each instant of time. 
For each opinion $i\in \{0,1\}$, we assume that there are $N_i^n =\ceil{n\gamma_i}$ stubborn agents whose opinion remains $i$ at all times, where  $\gamma_0,\gamma_1\in (0,1)$ satisfy $0<\gamma_0+\gamma_1<1$. The remaining agents are referred to as {\em non-stubborn agents} that update their opinions as follows. Each non-stubborn agent updates their opinion according to a unit-rate Poisson process. At each update, the agent samples $2k$ neighbours uniformly at random (with replacement) from its neighbourhood, with $k \geq 1$ fixed independently of $n$. The agent then updates its opinion by adopting the majority opinion among the group of sampled neighbours and itself.

The opinion dynamics can be described by the continuous-time process $X^n = (X^n(t), t\geq 0)$, where $X^n(t)$ denotes the proportion of agents holding opinion $1$ at time $t\geq 0$. Define $b_n = n/(n-1)$, the set $\mc S = [\gamma_1, 1-\gamma_0]$, and the polynomial $P_k:\mc S\to \RR$ as $P_k(x)=\sum_{r=k+1}^{2k}\binom{2k}{r} x^r \left( 1-x\right)^{2k-r}$. It is easy to see that the transition rates of the chain $X^n$ from state $x\in\mc{S}^n:=\{N_1^n/n, (N_1^n+1)/n,\ldots, 1-N_0^n/n\}$ to states $x\pm 1/n$, denoted respectively by $q_+^n(x)$ and $q_-^n(x)$, satisfy
\begin{align}
    &q^n_+(x) = nf^n_k(x),\label{eq:uprate}\\
    & q^n_-(x) = ng^n_k(x),\label{eq:downrate} 
\end{align}
where $f^n_k, g^n_k:\mc S\to \RR_+$ are given by
\begin{align}
    \nonumber f^n_k(x)&=\left(1-x - \frac{N_0^n}{n}\right)P_k(b_nx),\\
    \nonumber g^n_k(x)&=\left(x - \frac{N_1^n}{n}\right)P_k(b_n(1-x)).
\end{align}
Clearly, the process $X^n$ is ergodic. Hence, $X^n$ has a unique stationary distribution whose probability mass function we denote by $\pi^n$. Our interest lies in studying $\pi^n$ for large $n$ and how quickly (as a function of $n$) the distribution of $X^n$ converges to $\pi^n$ as this rate of convergence determines whether the network is able to maintain its proximity to the initial distribution of opinions (the law of $X^n(0)$) for a long time or not.

For our analysis, it will be useful to define the drift $\Delta^n_k:\mc S \to \RR$ and diffusion coefficient $\sigma_k^n : \mc S \to \RR_+$ as 
\begin{align}
    &\Delta_{k}^n(x) = {f^n_k(x)-g^n_k(x)},\nonumber\\
    &\sigma_k^n(x)=f_k^n(x)+g_k^n(x)\nonumber,
\end{align}
respectively. The drift $\Delta_k^n(x)$ captures the expected rate at which the proportion of agents holding opinion $1$ increases when the current state is $x\in \mc S^n$. Similarly, the diffusion coefficient $\sigma_k^n(x)$ captures the normalised jump rate out of a given state $x\in \mc S^n$. To characterise the rate at which a function $\phi:\mc S\to \RR$ changes along the trajectory of the process $X^n$, we will use the  generator operator $\mc G_{X^n}$ which, applied to the function $\phi$, takes the form
\begin{align}
    \mathcal{G}_{X^n} \phi(x) & = \lim_{h \to 0} \frac{\mathbb{E}[\phi(X^n(t+h))-\phi(X^n(t))|X^n(t) = x]}{h}\nonumber\\
    & \nonumber= q^n_+(x)(\phi(x+1/n) - \phi(x))+ q^n_-(x)(\phi(x-1/n) -\phi(x))\nonumber.
\end{align}
Thus, when the test function $\phi \in C^l(\mc S)$ is at least $l$ times continuously differentiable at $x\in \mc S$, using the Taylor series expansion of $\phi$ around $x$ in the above expression we obtain
\begin{align}
    &\mc G_{X^n} \phi(x)\leq\Delta_k^n(x)\sum_{i=0}^{\floor{(l-2)/2}} \frac{1}{n^{2i}(2i+1)!}\phi^{(2i+1)}(x)+\sigma_k^n(x)\sum_{i=1}^{\floor{(l-1)/2}}\frac{1}{n^{2i-1}(2i)!} \phi^{(2i)}(x)\nonumber\\
    & \hspace{1em}+\sigma_k^n(x)\frac{1}{n^{l-1}l!}\phi^{(l)}_{\max}\label{eq:generator expansion},
\end{align}
where $\phi^{(i)}$ denotes the $i^{\textrm{th}}$ derivative of $\phi$ and $\phi^{(i)}_{\max}=\sup_{x\in \mc S}|\phi^{(i)}(x)|$ for odd $i$ and $\phi^{(i)}_{\max}=\sup_{x\in \mc S}\phi^{(i)}(x)$ for even $i$.
Also define the limiting drift $\Delta_k:\mc S\to \RR$ and the limiting diffusion coefficient $\sigma_k:\mc S\to \RR_+$ as $\Delta_{k}(x) = f_k(x)-g_k(x)$ and 
$\sigma_{k}(x) = f_k(x)+g_k(x)$,
where 
\begin{align}
    f_k(x)&=\left(1-x - \gamma_0\right)P_k(x)\nonumber,\\
    g_k(x)&=\left(x - \gamma_1\right)P_k(1-x)\nonumber,
\end{align}
are the pointwise limits of $f_k^n$ and $g_k^n$, respectively. It is easy to see that there exists a constant $C>0$ such that $|f_k^n(x)-f_k(x)|<C/n, |g_k^n(x)-g_k(x)|<C/n, |\Delta_k^n(x)-\Delta_k(x)|< C/n$ for all $x\in \mc S$ and all sufficiently large $n$. In the lemma below (whose proof is given in Section \ref{appendix_sec:lemma1} of the Appendix), we characterise some important properties of the limiting drift  $\Delta_k$.

\begin{lemma}
\label{lem:Meta Lemma}
The limiting drift $\Delta_k$ satisfies the following properties for all $k\geq 1$.
\begin{enumerate}
    \item \label{prop:num_roots} $\Delta_k(\gamma_1)>0>\Delta_k(1-\gamma_0)$ and $\Delta_k$ has at least one and at most three real roots (counting multiplicity) in the interval $\mc S$.
    \item \label{prop:root_distance} Let $r_1$ and $r_2$ denote the smallest and the largest real roots of $\Delta_k$ in $\mc S$, respectively, with the possibility of $r_1=r_2$ when the root is unique. Then, for $\gamma_0\neq\gamma_1$ we have
    \begin{align}
        (\gamma_1-\gamma_0)(r_1+r_2-1+\gamma_0-\gamma_1) > 0,
        \label{eq:root_distance}
    \end{align}
    Similarly, for $\gamma_0=\gamma_1$, we have $r_1+r_2=1$ and we always have a root of $\Delta_k$ at $1/2$.
\end{enumerate}
\end{lemma}
Thus, according to the above lemma, depending on the values of the parameters $\gamma_0$ and $\gamma_1$, the limiting drift $\Delta_k$
can have either one, two or three real roots in the interval $\mc S$. This is a crucial property that holds for the polynomial $\Delta_k$, of degree $2k+1$ for all $k\geq 1$, and allows us to partition
the parameter space $\Gamma=\{(\gamma_0,\gamma_1)\in (0,1)^2: 0<\gamma_0+\gamma_1 <1\}$ of possible values of $(\gamma_0,\gamma_1)$ into three disjoint sets  $\Gamma_{\text{sub}}^k,\Gamma_{\text{crit}}^k,$ and $\Gamma_{\text{sup}}^k$ as defined below.
\begin{definition}
\label{def:set definitions}
    The parameter space $\Gamma=\{(\gamma_0,\gamma_1)\in (0,1)^2: 0<\gamma_0+\gamma_1 <1\}$ is partitioned into three disjoint sets $\Gamma_{\text{sub}}^k,\Gamma_{\text{crit}}^k,$ and $\Gamma_{\text{sup}}^k$ such that for $(\gamma_0,\gamma_1)\in \Gamma_{\text{sup}}^k$,  the polynomial $\Delta_k$ has exactly one real root (with multiplicity one) in $\mc S$, for $(\gamma_0,\gamma_1)\in \Gamma_{\text{sub}}^k$  the polynomial $\Delta_k$ has exactly three distinct real roots (each with multiplicity one) in the interval $\mc S$, and for $(\gamma_0,\gamma_1)\in \Gamma_{\text{crit}}^k$, the polynomial $\Delta_k$ has either one simple root and one double root, or one triple root.
\end{definition}
\begin{figure}
    \centering
    \includegraphics[width=0.7\linewidth]{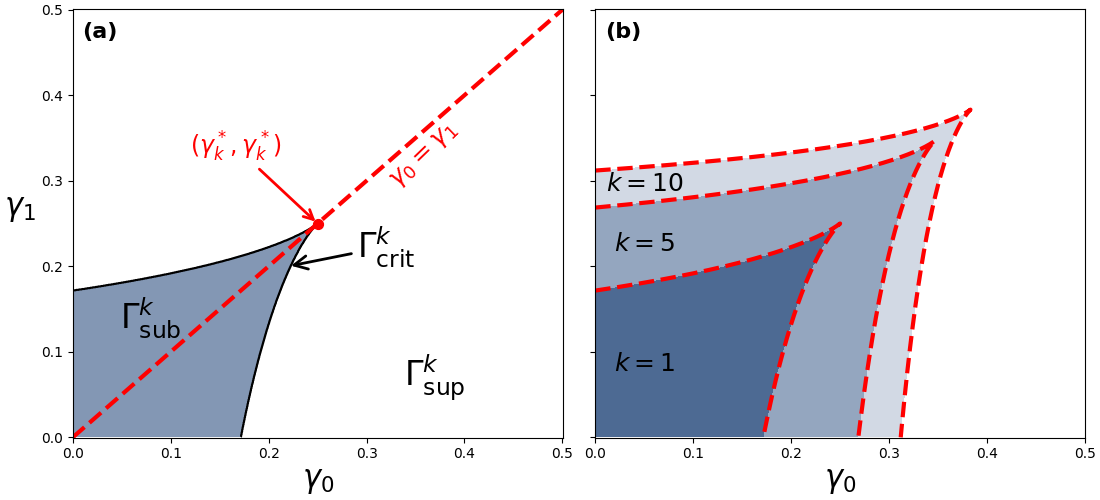}
    \caption{Plot of \textbf{(a)} the regions in the $(\gamma_0, \gamma_1)$ plane for which the function $\Delta_k$ (with $k=1$) has 3 roots (blue) and 1 root (white), and \textbf{(b)} the same plot for different values of $k$. In both plots, the regions were calculated numerically and red-dashed boundary lines displaying $\Gamma_{crit}^k$ were plotted using the parametric form in Proposition \ref{prop:parametric boundaries}, where the parameter ranges were solved numerically for each value of $k$.}
    \label{fig:fixed point regions}
\end{figure}
In Figure \ref{fig:fixed point regions}(a) we show these three regions for $k=1$.
We will see later that the dynamics of $X^n$ depend crucially on the region in which the parameters $(\gamma_0,\gamma_1)$ lie. Figure \ref{fig:fixed point regions}(b) shows how the regions change with $k$. In the proposition below we analytically characterise the region $\Gamma_{\text{crit}}^k$ for any $k\geq 1$. The proof of Proposition \ref{prop:parametric boundaries} is given in Appendix \ref{appendix_sec:prop3}.
\begin{proposition}
\label{prop:parametric boundaries}
    The parametric form of the boundary $\Gamma_{\text{crit}}^k$ is $$\{(\gamma_0(w),\gamma_1(w)), w\in (w_{\min}, w_{\max})\},$$ where
    \begin{align}
        &\nonumber \gamma_1(w)=w - \frac{P_k(w)(P_k(w)+P_{k}(1-w))}{P'_{k}(1-w)P_k(w)+P_k'(w)P_k(1-w)},
    \end{align}
    $\gamma_0(w)=\gamma_1(1-w)$, $w_{\min}$  solves $\gamma_0(w_{\min})=0$, and $w_{\max} = 1 - w_{\min}$. Here, $P'_k$ denotes differentiation with respect to the argument of $P_k$. 
    Furthermore, we have $\gamma_0(1/2)=\gamma_1(1/2)=\gamma_k^*$, where
    \begin{align}
        \gamma^*_{k} = \frac{1}{2}+\frac{1}{4k}-\frac{2^{2k}}{4k\binom{2k}{k}},
    \end{align}
    is the point at which  $\Gamma_{\text{crit}}^k$ intersects the line $\gamma_0=\gamma_1$. 
\end{proposition}

\section{Main Results}
\label{sec:Main Results}

In this section, we state and discuss the main results of the paper. Our first main result characterises the stationary distribution $\pi^n$ of the proportion of agents with opinion $1$ as a function of the proportions $\gamma_0$ and $\gamma_1$ of stubborn agents. We use $\Rightarrow$ to denote weak convergence and $\delta_x$ to denote the Dirac measure concentrated at $x\in \RR$.

\begin{theorem}
\label{thm:stationary}
Let $r_1$ and $r_2$ denote the smallest and largest real roots of $\Delta_k$ in $\mc S$, respectively, with the possibility of $r_1=r_2$ in case $\Delta_k$ has a unique root in $\mc S$.
Then, we have $\pi^n\Rightarrow \pi$ as $n\to \infty$ where $\pi$ is given by
        \begin{equation}
            \pi=\begin{cases}
                & \delta_{r_2}, \text{ if } \gamma_1 >\gamma_0,\\
                & \delta_{r_1}, \text{ if } \gamma_1 < \gamma_0,\\
                & \frac{1}{2}\brac{\delta_{r_1}+\delta_{r_2}}, \text{ if } \gamma_1 = \gamma_0.
            \end{cases}
        \end{equation}
\end{theorem}
The above theorem shows that when $\gamma_0\neq \gamma_1$, the stationary distribution of opinions tends to concentrate on a point where the majority opinion among the non-stubborn agents is the opinion with the higher proportion of stubborn followers. To see this, note that for $\gamma_1 >\gamma_0$ (resp. $\gamma_0>\gamma_1$), according to the above theorem, the stationary proportion of nodes with opinion $1$ among the entire population concentrates on $r_2$ (resp. $r_1$) and therefore the proportion of nodes with opinion $1$ among non-stubborn agents concentrates on $(r_2-\gamma_1)/(1-\gamma_0-\gamma_1)$ (resp. $(r_1-\gamma_1)/(1-\gamma_0-\gamma_1)$) which by the second statement of Lemma~\ref{lem:Meta Lemma} satisfies  $(r_2-\gamma_1)/(1-\gamma_0-\gamma_1)\geq(r_1+r_2-2\gamma_1)/2(1-\gamma_0-\gamma_1)> 1/2$ (resp. $(r_1-\gamma_1)/(1-\gamma_0-\gamma_1)\leq(r_1+r_2-2\gamma_1)/2(1-\gamma_0-\gamma_1)< 1/2$).
Hence, at steady-state, the majority of non-stubborn agents follow the opinion which has the higher proportion of stubborn followers. This is natural to expect given that the update rule favours the opinion having the higher number of followers. In contrast, when the proportions of stubborn followers of the two opinions are equal, i.e., $\gamma_0=\gamma_1$, the stationary distribution is symmetric around $1/2$ since in this case we have $r_1=1-r_2$ (as shown in Lemma~\ref{lem:Meta Lemma}, meaning that no opinion strictly dominates at the steady-state.



The result of Theorem~\ref{thm:stationary} seems to indicate that the long-term behaviour of the non-stubborn agents is determined primarily by the larger stubborn group. However, our next set of results show that the speed at which the distribution of $X^n(t)$ converges to $\pi^n$ varies drastically depending on the proportions $\gamma_0$ and $\gamma_1$ of stubborn agents and, as a result, the non-stubborn agents may (in case of slow convergence to $\pi^n$) be able to retain their initial majority opinion for a long period of time before their majority opinion switches to the one with the higher number of stubborn followers. 

To state our results, for any $p\in\mc S$ and any $\delta>0$, we denote $B_{\delta}(p)=(p-\delta,p+\delta)$ and for any $B\subseteq \mc S$ we define $T_{\text{exit}}^n(B)=\inf\{t\geq 0: X^n(t)\notin B\}$ and $T_{\text{enter}}^n(B)=\inf\{t\geq 0: X^n(t)\in B\}$ as the first times the chain $X^n$ exits and enters the set $B$, respectively. For $p\in \mc S$ we denote by $\PP_p(\cdot)$ and $\EE_p[\cdot]$ the probability and expectation, respectively, conditioned on $X^n(0)=\ceil{np}/n$. 
Furthermore, we define the {\em maximum total variation distance} between the distribution of $X^n(t)$ and $\pi^n$ as $d^n(t) = \max_{x \in \mc{S}^n}\max_{\mathcal{X} \in \mc{P}(\mc{S}^n)}|\PP_{x}(X^n(t) \in \mc{X}) - \pi^n(\mc{X})|$ where $\mc P(\mc S^n)$ denotes the power set of $\mc S^n$. For any $\epsilon>0$, we define $t_{\text{mix}}^n(\epsilon)=\inf\{t:d^n(t)\leq \epsilon\}$. Thus, $t_{\text{mix}}^n(\epsilon)$ denotes the first time the maximum total variation distance between the distribution of $X^n(t)$ and $\pi^n$ falls below a threshold of $\epsilon$. In particular, for $\epsilon=1/4$, we define $t_{\text{mix}}^n:=t_{\text{mix}}^n(1/4)$. Thus, a higher value of $t_{\text{mix}}^n$ indicates a slower convergence to the stationary distribution $\pi^n$. We say that an event occurs ``with high probability'' or in short ``w.h.p.'' when its probability of occurrence converges to one as $n\to \infty$. 


The next theorem characterises the process $X^n$ when $(\gamma_0, \gamma_1) \in \Gamma_{\text{sub}}^k$. 
\begin{theorem}
\label{thm:slow_mixing}
    Fix $(\gamma_0,\gamma_1)\in \Gamma_{\text{sub}}^k$ and let $r_1< r_m < r_2$ denote the three distinct real roots (each having a multiplicity of one) of $\Delta_k$ in $\mc S$. Then we have $t_{\text{mix}}^n = \Omega(\exp (\Theta(n)))$. Furthermore, for any  constant $p\neq r_m$ and any sufficiently small $\delta,\epsilon>0$, we have $\PP_p(T_{\text{enter}}^n(B_{\delta}({r^*}))>C\log n/\epsilon)\leq \epsilon$ for some positive constant $C$ and for all sufficiently large $n$  where $r^*=r_1\indic{p<r_m}+r_2\indic{p>r_m}$. For any $i\in \{1,2\}$, we also have $T_{\text{exit}}^n(B_\delta(r_i))=\Omega(\exp(\Theta(n)))$ w.h.p. for any starting state $X^n(0)\in B_{\delta/2}(r_i)$.
\end{theorem}
\begin{figure}
    \centering
    \includegraphics[width=0.8\linewidth]{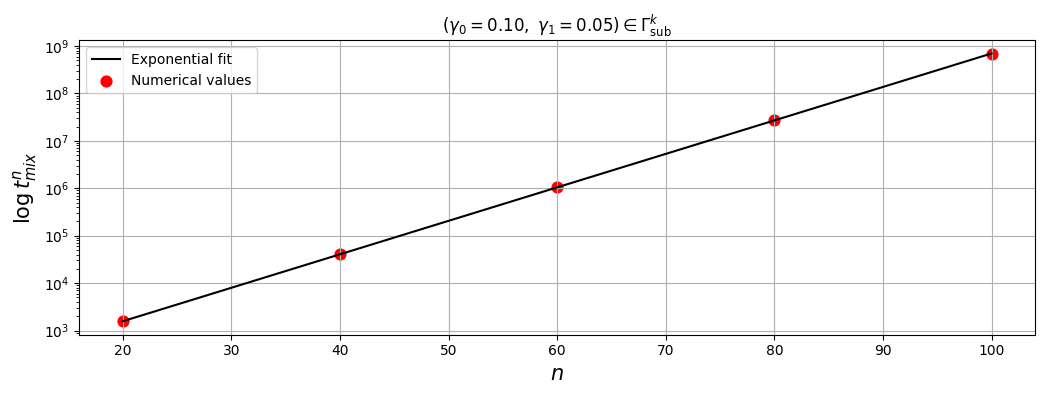}
    \caption{Plot of the logarithm of the mixing time of the process against $n$ in sub-critical parameter regime $(\gamma_0, \gamma_1) \in \Gamma_{sub}^{k}$. To compute the mixing time, the transient distribution of $X^n(t)$ was calculated by solving the forward Kolmogorov equation, and the stationary distribution $\pi^n$ was calculated by solving the detailed balance equations.}
    \label{fig:numerical_validation_sub_critical}
\end{figure}
\textit{Discussion on the result:} This result demonstrates that when $(\gamma_0, \gamma_1) \in \Gamma_{\text{sub}}^k$, the dynamics' mixing time grows exponentially with $n$ which we numerically validate in Figure \ref{fig:numerical_validation_sub_critical}. Specifically, if the system starts with a proportion $p$ of agents holding opinion $1$, the chain $X^n(t)$ rapidly converges in just $O(\log n)$ time to whichever root among $r_1$ and $r_2$ of $\Delta_k$ lies on the same side of $r_m$ as $p$. After reaching one of the roots $r_1$ or $r_2$, the chain remains trapped near this root for an exponentially long duration before switching to the other root. 

Thus, in the sub-critical regime $\Gamma_{\text{sub}}^k$, the opinion that initially holds the majority among non-stubborn agents may continue to do so for an exponentially long period even if the opposite opinion has a larger proportion of stubborn followers. As shown in Figure~\ref{fig:fixed point regions}, the region  $\Gamma_{\text{sub}}^k$ corresponds to small values of both $\gamma_0$ and $\gamma_1$ and a small difference between them. Intuitively, when stubborn populations are small and balanced, their influence is negligible, leaving the majority rule as the primary driver of the network dynamics. 

\textit{Intuitive explanation using the drift:} In Figure \ref{fig:drift_sub_critical}, we see that the limiting drift $\Delta_k$ has three simple roots, $r_1, r_m$, and $r_2$ in $\mc{S}$ when $(\gamma_0,\gamma_1)\in \Gamma_{\text{sub}}^k$. Furthermore, the two extreme roots $r_1$ and $r_2$ are both {\em stable} in the sense that for each of these roots there is an interval $[a,b]$ containing the root such that $\Delta_k(x)>0$ for all $x\in [a,r)$ and $\Delta_k(x)<0$ for all $x\in (r,b]$. This implies that the proportion of agents with opinion $1$ increases when the process $X^n$ is below the stable root and decreases when the process is above the stable root. Hence, the process $X^n$ is driven towards the stable root from either side. The middle root $r_m$ is {\em unstable} in the sense that when the process $X^n$ is sufficiently close (but not equal) to $r_m$, the drift takes it further away from this root. We show that this nature of the drift results in the process $X^n$ visiting the stable root closest to its current state exponentially many times before crossing the unstable root.

\begin{figure}
    \centering
    \includegraphics[width=0.8\linewidth]{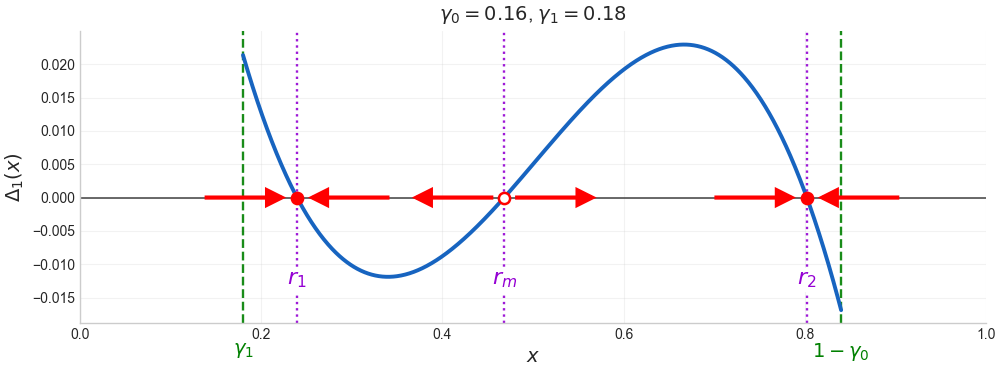}
    \caption{Plot of the limiting drift $\Delta_k$ for $(\gamma_0, \gamma_1) \in \Gamma_{\text{sub}}^k$, with $k=1$. The red arrows indicate the direction of the drift around the roots.}
    \label{fig:drift_sub_critical}
\end{figure}

Next, we analyse the process $X^n$ when $(\gamma_0, \gamma_1) \in \Gamma_{\text{sup}}^k$. Our main result for this regime is the following.

\begin{theorem}
    \label{thm:fast_mixing}
    Fix $(\gamma_0, \gamma_1)\in\Gamma_{\text{sup}}^k$ and let $r$ be the unique simple root of $\Delta_k$ in $\mc S$. Then, for any $\delta>0$ and any $x \in \mc S^n$, we have $\EE_x[T_{enter}^n(B_{\delta}(r))]=O(\log n)$. Furthermore, the mixing time of the process is $t_{mix}^n=O(\log n)$.
\end{theorem}
\begin{figure}[t]
    \centering
    \includegraphics[width=0.8\linewidth]{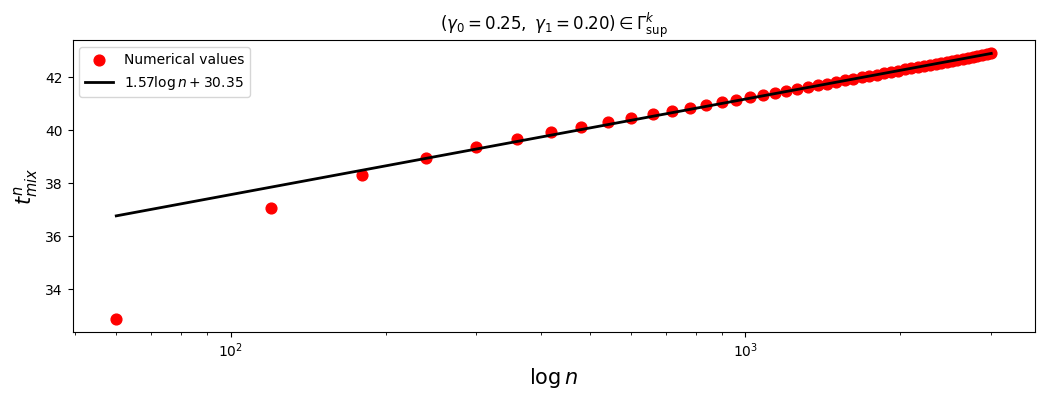}
    \caption{Plot of the mixing time of the process against $\log n$ in super-critical parameter regime $(\gamma_0, \gamma_1) \in \Gamma_{sup}^{k}$. To compute the mixing time numerically, the transient distribution of $X^n(t)$ was calculated by solving the forward Kolmogorov equation, and the stationary distribution $\pi^n$ was calculated by solving the detailed balance equations.}
    \label{fig:numerical_validation_super_critical}
\end{figure}
\textit{Discussion on the result:} The result above shows that when $(\gamma_0, \gamma_1) \in \Gamma_{\text{sup}}^k$, the dynamics converge arbitrarily close to the unique root $r$ of $\Delta_k$ in $\mathcal{S}$ in only $O(\log n)$ time. 
This is numerically validated in Figure \ref{fig:numerical_validation_super_critical}.
From Theorem~\ref{thm:stationary}, $r$ is also the point where the stationary distribution $\pi^n$ concentrates as $n \to \infty$. Consequently, the mixing time of the dynamics is also $O(\log n)$.
This implies that  the majority opinion among non-stubborn agents rapidly shifts toward the opinion with the larger number of stubborn followers, regardless of the network's initial state. Intuitively, as illustrated in Figure~\ref{fig:fixed point regions}, at least one of the stubborn proportions ($\gamma_0$ or $\gamma_1$) is large in $\Gamma_{\text{sup}}^k$. This stronger stubborn presence provides enough momentum to swiftly overcome any initial bias in the population. 
\begin{figure}
    \centering
    \includegraphics[width=0.8\linewidth]{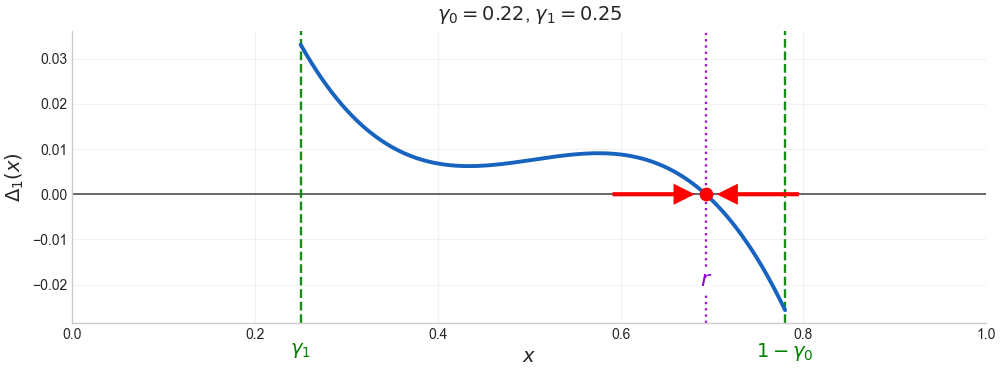}
    \caption{Plot of the limiting drift $\Delta_k$ for $(\gamma_0, \gamma_1) \in \Gamma_{\text{sup}}^k$, with $k=1$. The red arrows indicate the direction of the drift around the roots.}
    \label{fig:drift_sup_critical}
\end{figure}

\textit{Intuitive explanation using the drift:} In Figure \ref{fig:drift_sup_critical}, we see that the unique simple root $r$ of $\Delta_k$ in $\mc{S}$ is globally stable in the sense that for all $x < r$ the drift $\Delta_k$ is positive, and for all $x > r$, the drift $\Delta_k$ is negative. 
As a result, the limiting drift geometrically decreases the distance $|X^n(t)-r|$.
This causes the chain to converge to the root $r$ exponentially fast from anywhere in the state-space.


We now turn to the regime where $(\gamma_0, \gamma_1) \in \Gamma_{\text{crit}}^k$. The analysis of the process $X^n$ is most challenging at this critical boundary since at the boundary the drift $\Delta_k$ has a root with multiplicity higher than one in $\mc S$; for $(\gamma_0, \gamma_1) \in \Gamma_{\text{crit}}^k\cap\{\gamma_0=\gamma_1\}$ there exists a unique root at $1/2$ with multiplicity three (which follows from the second statement of Lemma~\ref{lem:Meta Lemma} and the definition of $\Gamma_{\text{crit}}^k$) and for $(\gamma_0, \gamma_1) \in \Gamma_{\text{crit}}^k\cap\{\gamma_0\neq \gamma_1\}$ there are two distinct roots, one with multiplicity one and the other with multiplicity two (which again follows from the second statement of Lemma~\ref{lem:Meta Lemma} and the definition of $\Gamma_{\text{crit}}^k$). As shown below, the existence of such roots significantly alters the dynamics as compared to the sub- and super-critical regimes.  

\begin{theorem}
\label{thm:criticality}
    Fix $(\gamma_0, \gamma_1)\in \Gamma_{\text{crit}}^k$. The two possibilities are $\gamma_0=\gamma_1$ and $\gamma_0\neq\gamma_1$. For $\gamma_0=\gamma_1$, let $r$ be the unique root of $\Delta_k$ in $\mc S$ with multiplicity three and for $\gamma_0\neq\gamma_1$, let $r$ denote the root of $\Delta_k$ in $\mc S$ with multiplicity one. The following statements hold.
    \begin{enumerate}
        \item If $\gamma_0 \neq \gamma_1$, then, for each sufficiently small $\delta>0$ and $x\in \mc S^n$, we have $\EE_{x}[T_{\text{enter}}^n(B_{\delta}(r))]= O(n^{1/3})$, and the mixing time of the process $X^n$ satisfies $t_{\text{mix}}^n =\Theta(n^{1/3})$.
        \item If $\gamma_0=\gamma_1$, then for any $x \in \mc S^n$, we have $\EE_x[T^n_{\text{enter}}(B_{1/n}(r))]= O(n^{1/2})$ and the mixing time of the process $X^n$ satisfies $t_{\text{mix}}^n =\Theta(n^{1/2})$.
    \end{enumerate}
\end{theorem}

\begin{figure}[t]
    \centering
    \includegraphics[width=0.8\linewidth]{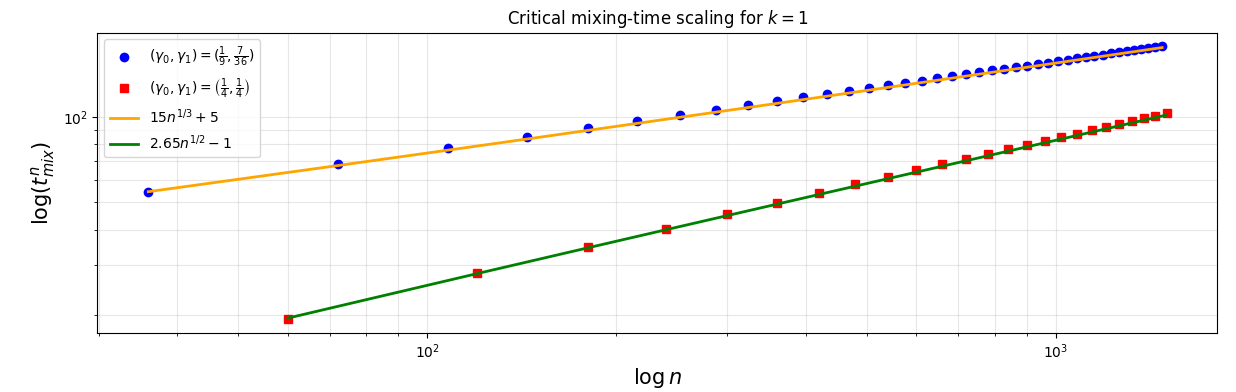}
    \caption{Plot of the logarithm of the mixing time of the process against $\log n$ in sub-critical parameter regime $(\gamma_0, \gamma_1) \in \Gamma_{crit}^{k}$. To compute the mixing time numerically, the transient distribution of $X^n(t)$ was calculated by solving the forward Kolmogorov equation, and the stationary distribution $\pi^n$ was calculated by solving the detailed balance equations.}
    \label{fig:numerical_validation_critical}
\end{figure}

\textit{Discussion on the result:} The above result shows that when $(\gamma_0, \gamma_1) \in \Gamma_{\text{crit}}^k$, unlike the sub- and super-critical regimes, the mixing time of the dynamics grows polynomially with $n$. This is numerically validated in Figure \ref{fig:numerical_validation_critical}. Thus, the non-stubborn agents may retain the majority of the initially prevalent opinion for a duration that grows polynomially with $n$. Although this is smaller than the exponential retention window in the sub-critical case, it is still significantly larger than the logarithmic retention window at supercriticality.

\begin{figure}[!t]
    \centering
    \includegraphics[width=0.8\linewidth]{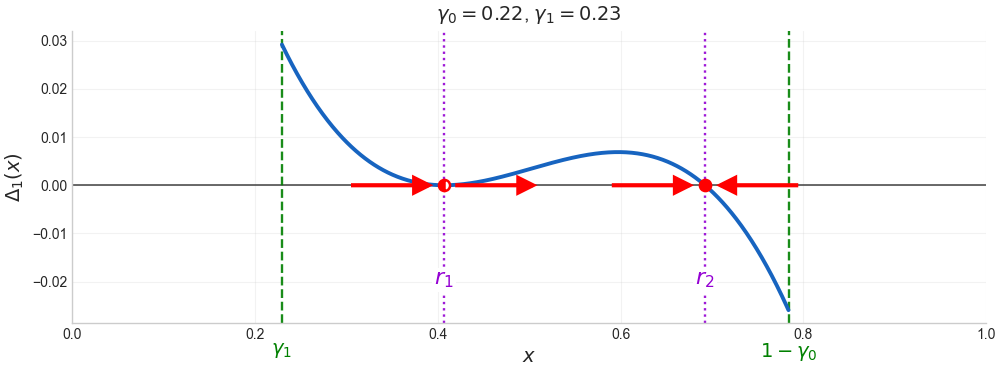}
    \caption{Plot of the limiting drift $\Delta_k$ for $(\gamma_0, \gamma_1) \in \Gamma_{\text{crit}}^k$, with $k=1$. The red arrows indicate the direction of the drift around the roots.}
    \label{fig:drift_critical}
\end{figure}

\textit{Intuitive explanation using the drift and diffusion:} In Figure \ref{fig:drift_critical}, the limiting drift $\Delta_k$ is plotted for $(\gamma_0, \gamma_1) \in \Gamma_{\text{crit}}^k\cap\{\gamma_0< \gamma_1\}$. 
In this case, there exist two distinct real roots of $\Delta_k$ in $\mc S$ of which the higher root $r_2$ has a multiplicity of one and the lower root $r_1$ has a multiplicity of two. We also note that $r_2$ is a stable root since on either side of this root the drift points to the root. However, to reach a neighbourhood of the root $r_2$ from below $r_1$ the process $X^n$ needs to cross the degenerate root $r_1$ where the drift becomes zero. Hence, near the degenerate root $r_1$ the drift alone is not sufficient to push the process $X^n$ toward $r_2$. Near $r_1$, the process is mainly driven by stochastic fluctuations arising from the diffusion term $\sigma_k$. Indeed, choosing $\phi(x)=(x-r_1)^2$ in \eqref{eq:generator expansion} and observing that $\Delta_k(x)\approx c (x-r_1)^2$ for some $c>0$ near the root $r_1$ which has a multiplicity of two we obtain $G_{X^n}\phi(x)\approx c(x-r_1)^3+\sigma_k(r_1)/n$ for all $x$ near the root $r_1$. In particular, when $|x-r_1|\leq 1/n^{1/3}$, the first term $(x-r_1)^3$ becomes $O(1/n)$ which is comparable  in order to the second term $\sigma_k(r_1)/n$. Hence, the rate at which the variance $(x-r_1)^2$ increases in the region $|x-r_1|\leq 1/n^{1/3}$ is $O(1/n)$.
Therefore, crossing the region $|x-r_1|\leq 1/n^{1/3}$ requires at least $\Omega((1/n^{1/3})^2/(1/n))=\Omega(n^{1/3})$ time on average since $\phi$ needs to change by at least $\Omega((1/n^{1/3})^2)$ to cross this region. A similar explanation can be given for the case $\gamma_0=\gamma_1$ where  the drift behaves as $\Delta_k(x)\approx-c(x-r)^3$ for some $c>0$ near the unique root $r=1/2$ and therefore to cross a region of width $\Theta(1/n^{1/4})$ around $r$ it takes at least $\Omega(n^{1/2})$ time on average.

\section{A central lemma}
\label{Sec:Meta Lemma}

Before delving into the proofs of our main results, we state and prove a lemma which is central in all our proofs. The lemma holds for any one-dimensional  birth-death process $X^n$ defined on $\mc S$ whose limiting inifinitesimal drift $\Delta_k$ is a polynomial. We emphasize here that the lemma does not require any global monotonicity property of the limiting inifinitesimal drift $\Delta_k$ and hence is applicable to non-monotone systems.

\begin{lemma}
\label{lem:log_hitting_time}
Suppose that the polynomial $\Delta_k$ has a  simple root $r\in \mc{S}$ that is stable, i.e., there exists an  interval (region of attraction) $[a,b]\subseteq \mc S$ containing the root such that $\Delta_k(x) >0$ for $x\in [a,r)$ and $\Delta_k(x) < 0$ for $x\in (r,b]$. Then the following properties hold.
    \begin{enumerate}

        \item There exist constants $\eta,C_1>0$ such that for all $x\in B_{\eta}(r)$ we have $\Delta_k'(x)<-C_1$.
        Furthermore, there exists a constant $C_2>0$ such that for all $x \in [a,b]$, we have $\Delta_k(x) (x-r)\leq -C_2 (x-r)^2$.

        \item For any $\delta>0$ satisfying $B_\delta(r)\subseteq [a,b]$, there exists $\kappa_\delta>0$  such that $\PP_{x}(T_{\text{exit}}^n(B_{\delta}(r))\leq\exp(\kappa_\delta n))\leq C_4n\exp(-\kappa_\delta n/4)$ for all sufficiently large $n$ and for any $x \in B_{\delta/2}(r)$ where $C_4>0$ is a constant independent of $n$.
        
        \item For any $\delta>0,\epsilon>0$, there exists $n_0(\delta,\epsilon)>0$ such that whenever $n\geq n_0(\delta,\epsilon)$ we have $\EE_{x}[T^n_{\text{enter}}(B_{\delta}(r))\wedge T^n_{\text{exit}}([a,b])]\leq C\log n$ and $\PP_x(T^n_{\text{enter}}(B_\delta(r))>C'\log n/\epsilon)\leq \epsilon$ for all initial state $x\in [a,b]$ where $C$ and $C'$ are positive constants independent of $n$.

        \item  If $X^n$ and $\bar X^n$ are chains evolved independently until their first meeting time and synchronously thereafter, started at $x$ and $\bar x$, respectively, and each following the transition rates~\eqref{eq:uprate}-\eqref{eq:downrate}, then for any $\epsilon>0$ and all $n$ sufficiently large, we have $\PP_{x,\bar x}(X^n(t)\neq \bar X^n(t))\leq \epsilon$ whenever $t> \frac{C_3\log n}{\epsilon}$ for some positive constant $C_3$ and  $x,\bar x\in B_{\eta}(r)$ for sufficiently small $\eta>0$.
    \end{enumerate} 
\end{lemma}


\subsection{Proof of the first statement}

Since $\Delta_k$ has exactly one real root $r$ with multiplicity one in $[a,b]$, we must have $\Delta_k(x)=(x-r)R_k(x)$ for some polynomial $R_k$ such that $R_k(r)\neq 0$. Furthermore, since we have $\Delta_k(x)>0$ for $x\in [a, r)$ and $\Delta_k(x)<0$ for $x\in (r,b]$ we must have $R_{k}(x) < 0$ for all $y\in [a,b]-\{r\}$. The continuity of $R_k$ at $r$ combined with the fact $R_k(r)\neq 0$ then implies that we must have $\Delta'(r)=R_k(r)<0$. Let $\Delta_k'(r)=-2C_1$ where $C_1>0$ is some constant. Then, the continuity of $\Delta'_{k}$ at $r$ implies that there exists $\eta>0$ such that for all $y\in B_{\eta}(r)$ we have $\Delta'_{k}(y)< -C_1$.
Now, we define the map $x\mapsto T(x)$ on the compact set $[a,b]$ as
    \begin{equation}
        T(x)=\begin{cases}
            &\frac{\Delta_{k}(x)}{x-r} \text{ for } x\neq r,\\
            &\Delta_{k}'(x) \text{ for } x= r\nonumber.
        \end{cases}
    \end{equation}
Clearly, $T$ is strictly negative and continuous on the compact set $[a,b]$. Hence, there exists $C_2>0$ such that $T(x) \leq-C_2$ for all $x\in [a,b]$. Hence, we have $\Delta_k(x)\leq -C_2(x-r)$ for all $x\in (r,b]$ and $\Delta_k(x)\geq C_2(r-x)$ for $x\in [a,r)$. This completes the proof of the first statement.

\subsection{Proof of the second statement}

Consider the Lyapunov function $\phi^n_\theta(x)=\exp\brac{n\frac{\theta}{2}(x-r)^2}$ for $\theta\in (0,\bar\theta)$. 
It is easy to see that
\begin{equation*}
    \phi^n_\theta\left(x\pm \frac{1}{n}\right)=\phi^n_\theta(x)\exp\brac{\pm\theta(x-r)+\frac{\theta}{2n}}.
\end{equation*}
Hence, applying the generator $\mc G_{X^n}$ to $\phi^n_\theta$ we obtain
\begin{align}
    G_{X^n}\phi^n_\theta(x)&=nf_k^n(x)\phi^n_\theta(x)\brac{\exp\brac{\theta(x-r)+\frac{\theta}{2n}}-1}\nonumber\\
    &\hspace{2em}+ng_k^n(x)\phi^n_\theta(x)\brac{\exp\brac{-\theta(x-r)+\frac{\theta}{2n}}-1}.\nonumber
\end{align}
Now, using the fact that $\exp(z)\leq1+z+C_M z^2$ for some $C_M>0$ for all $z\in [-M,M]$
and the definitions of $\Delta_k^n$ and $\sigma_k^n$ we obtain for all sufficiently large $n$
\begin{align}
    G_{X^n}\phi^n_\theta(x)\leq n\phi^n_\theta(x)\brac{\Delta_k^n(x)\theta(x-r)+\frac{\theta}{2n}\sigma_k^n(x)+C_{\bar \theta,\delta}\theta^2},
    \label{eq:gen_exp_bound}
\end{align}
for some constant $C_{\bar \theta,\delta}>0$ since $\theta|x-r|+\theta/2n\leq \bar2\theta\delta$ for $x\in B_{\delta}(r)$ and $n$ sufficiently large.
But from the first part of the lemma and the fact that $|\Delta_k^n(x)-\Delta_k(x)|< \frac{C}{n}$
we can bound the first term in the parenthesis on the RHS as $\Delta_k^n(x)(x-r)\leq \Delta_k(x)(x-r)+\frac{C}{n}(x-r)\leq -C_2 (x-r)^2+\frac{C}{n}(x-r)\leq -\frac{C_2}{2}(x-r)^2$ where to obtain the last inequality we assume $|x-r|\geq \eta$ for some $\eta \in (0,\delta/2)$ and choose $n>\frac{2C}{\eta C_2}$. Similarly, since $\sigma_k^n$ converges uniformly to $\sigma_k$ and $\sigma_{\min}<\sigma_k(x)<\sigma_{\max}$ for all $x\in \mc S$ for some $\sigma_{\min},\sigma_{\max}>0$ it follows that for all sufficiently large $n$ we have $\sigma_{\min} < \sigma_k^n(x)<\sigma_{\max}$. Hence, from~\eqref{eq:gen_exp_bound} we obtain
\begin{align}
  G_{X^n}\phi^n_\theta(x)\leq  n\phi^n_\theta(x)\brac{-\frac{C_2\theta}{2}(x-r)^2+\frac{\theta}{2n}\sigma_{\max}+C_{\bar \theta,\delta}\theta^2}.\nonumber
\end{align}
Hence, for $|x-r|\geq\eta$ we obtain
\begin{align*}
    G_{X^n}\phi^n_\theta(x)\leq  n\phi^n_\theta(x)\brac{-\frac{C_2\theta}{2}\eta^2+\frac{\theta}{2n}\sigma_{\max}+C_{\bar \theta,\delta}\theta^2}.
\end{align*}
Now choosing $n> \frac{2\sigma_{\max}}{C_2 \eta^2}$ and $\theta<\frac{C_2\eta^2}{4C_{\bar \theta,\delta}}$ we obtain
$$G_{X^n}\phi^n_\theta(x)\leq  n\phi^n_\theta(x)\brac{-\frac{C_2\theta}{4}\eta^2+C_{\bar \theta,\delta}\theta^2}<0.$$
Thus, for suitable choice of $\theta\in \brac{0,\bar\theta \wedge \frac{C_2\eta^2}{4C_{\bar \theta,\delta}}}$, we have
$G_{X^n}\phi^n_\theta(x)\leq 0$ for all $x\in B_{\delta}(r)-B_{\eta}(r)$ and all sufficiently large $n$. This implies that $(\phi^n_\theta(X^n(t\wedge T^n_\text{enter}(B_\eta(r))\wedge T^n_\text{exit}(B_\delta(r)))),t\geq 0)$ is a super-martingale. Therefore, applying the optional sampling theorem we obtain $$\EE_x\sbrac{\phi^n_\theta(X^n(T^n_\text{enter}(B_\eta(r))\wedge T^n_\text{exit}(B_\delta(r))))}\leq \phi^n_\theta(x)\leq \exp\brac{n\frac{\theta\delta^2}{8}},$$ for any starting state $x\in B_{\delta/2}(r)$. But 
$$\EE_x\sbrac{\phi^n_\theta(X^n(T^n_\text{enter}(B_\eta(r))\wedge T^n_\text{exit}(B_\delta(r))))}\geq \exp\brac{n\frac{\theta\delta^2}{2}}\PP_x\brac{T^n_{\text{exit}}(B_{\delta}(r))\leq T^n_{\text{enter}}(B_\eta(r))}.$$
Hence, we have
\begin{equation}
    \PP_x\brac{T^n_{\text{exit}}(B_{\delta}(r))\leq T^n_{\text{enter}}(B_\eta(r))}\leq \exp\brac{-n\kappa_\delta}, \quad \forall x\in B_{\delta/2}(r)-B_\eta(r)\nonumber,
\end{equation}
where $\kappa_\delta=\frac{3\theta\delta^2}{8}$.
Hence, if $N$ denotes the number of times the chain enters the set $B_{\eta}(r)$ before exiting $B_{\delta}(r)$ starting from $x\in B_{\delta/2}(r)$, then $\PP_x(N\leq m)=1-\PP_x(N> m)\leq 1-(1-\exp(-\kappa_\delta n))^{m+1}\leq (m+1)\exp(-\kappa_\delta n)$. Let $J(t)$ denote the number of times the process $X^n(t)$ jumps by time $t$. Since the total jump rate is bounded above by $n\sigma_{\max}$, we have $\EE_x[J(t)]\leq n\sigma_{\max}t$.
Now since for any $m\in \NN$, we have $\{T^n_{\text{exit}}(B_\delta(r))\leq t\}\subseteq \{N\leq J(t)\}\subseteq \{N\leq m\}\cup \{J(t)\geq m\}$ it follows that
$\PP_x(T^n_{\text{exit}}(B_\delta(r))\leq t)\leq \PP_x(N\leq m)+\PP_x(J(t)\geq m)\leq (m+1)\exp(-\kappa_\delta n)+\EE_{x}[J(t)]/m\leq (m+1)\exp(-\kappa_\delta n)+n\sigma_{\max}t/m$. Now choosing $t=\exp(\frac{\kappa_\delta n}{4})$ and $m=\ceil{\exp(\kappa_\delta n/2)}$ gives 
\begin{equation}
\label{eq:exit time bound}
\PP_x\left(T^n_{\text{exit}}(B_\delta(r))\leq \exp\left(\frac{\kappa_\delta n}{4}\right)\right)\leq n\sigma_{\max}\exp\left(-\frac{\kappa_\delta n}{4}\right)+2\exp(-\kappa_\delta n)+\exp\left(-\frac{\kappa_\delta n}{2}\right),
\end{equation}
from which the second statement follows.

\subsection{Proof of the third statement} 

We consider the drift of the Lyapunov function $\phi:\mc S\to \RR$ defined as 
$\phi(x)=\log(1+n(x-r)^2)$. Clearly, we have for $|x-r|\geq \delta$
\begin{align*}
    \phi'(x)&=\frac{2n(x-r)}{1+n(x-r)^2}\\
    \phi''(x)&=\frac{2n(1-n(x-r)^2)}{(1+n(x-r)^2)^2}\leq \frac{2n}{(1+n(x-r)^2)}\leq \frac{2n}{1+n\delta^2}\leq \frac{2}{\delta^2}.
\end{align*}
Hence, applying the generator $G_{X^n}$ on $\phi$ and using the above bound we have for $x\in [a,b]-B_\delta(r)$
\begin{align}
G_{X^n}\phi(x)&\leq \Delta_k^n(x)\frac{2n(x-r)}{1+n(x-r)^2}+\frac{\sigma_k^n(x)}{2n} \frac{2}{\delta^2}\nonumber\\
              &\leq \Delta_k(x)\frac{2n(x-r)}{1+n(x-r)^2}+\frac{2C(x-r)}{1+n(x-r)^2}+\frac{\sigma_{\max}}{n\delta^2}\nonumber\\
              &\overset{(a)}{\leq} -\frac{2nC_2(x-r)^2}{1+n(x-r)^2}+\frac{2C|x-r|}{1+n(x-r)^2}+\frac{\sigma_{\max}}{n\delta^2}\nonumber\\
              &\overset{(b)}{\leq} -\frac{2nC_2(x-r)^2}{2n(x-r)^2}+\frac{2C|x-r|}{n(x-r)^2}+\frac{\sigma_{\max}}{n\delta^2}\nonumber\\
              &\leq -C_2+\frac{2C}{n\delta}+\frac{\sigma_{\max}}{n\delta^2}\nonumber\\
              &\leq -\frac{C_2}{2},\label{eq:negative_drift}
\end{align}
where (a) follows from the first statement of the lemma, (b) follows by choosing  $n\geq 1/\delta^2$ so that $n(x-r)^2\geq n\delta^2\geq 1$, and the last line again follows by choosing $n$ sufficiently large.
Hence, defining $\tau=T^n_{\text{enter}}(B_\delta(r))\wedge T^n_{\text{exit}}([a,b])$ and applying Dynkin's formula~\cite{Kallenberg_book} we obtain for any $t\geq 0$ 
\begin{align*}
\EE_x[\phi(X^n(t\wedge \tau))]&=\EE_x[\phi(X^n(0))]+\EE_x\sbrac{\int_0^{t\wedge \tau}\mc G_{X^n}\phi(X^n(s))ds}\\
&\leq \log(1+n(b-a)^2)-\frac{C_2\EE_x[t\wedge \tau]}{2},
\end{align*}
where in the last inequality we have used $\phi(x)\leq \log(1+n(b-a)^2)$ for any $x\in [a,b]$ and
\eqref{eq:negative_drift}. Now, using $\EE_x[\phi(X^n(t\wedge \tau))]\geq 0$ we obtain
$\EE_x[t\wedge \tau]\leq (2/C_2)\log(1+n(b-a)^2)$. Since this is true for all $t\geq 0$, using dominated convergence theorem it follows that $\EE_x[\tau]\leq (2/C_2)\log(1+n(b-a)^2)\leq C'\log n$ for some $C'>0$ and all sufficiently large $n$ which proves the expectation bound in the statement. To complete the proof we note that using Markov inequality gives $\PP_x(\tau\geq C'\log n/\epsilon)\leq \epsilon$.
But we have
\begin{align*}
\PP_x(\tau\geq t)&=\PP_x(T^n_{\text{enter}}(B_\delta(r))\geq t, T^n_{\text{exit}}([a,b])\geq t)\\
& \geq \PP_x(T^n_{\text{enter}}(B_\delta(r))\geq t)- \PP_x(T^n_{\text{exit}}([a,b])< t).
\end{align*}
Now choosing $t=C'\log n/\epsilon$ and using the previous part of the lemma we obtain
$\PP_x(\tau\geq C'\log n/\epsilon)\leq \PP_x(T^n_{\text{enter}}(B_\delta(r))\geq C'\log n/\epsilon)- o(1)$. Thus, $\PP_x(T^n_{\text{enter}}(B_\delta(r))\geq C'\log n/\epsilon)\leq \epsilon+o(1)$ from which the statement of the lemma follows immediately.

\subsection{Proof of the last statement} 

Consider the joint chain $(X^n,\bar X^n)$ and assume without loss of generality that $X^n(0)=x\leq \bar x =\bar X^n(0)$. Furthermore, assume that both $x$ and $\bar x$ are in $B_{\eta}(r)$ where $\eta>0$ is such that $\Delta_k'(x)<-C_2$ for all $x\in B_{2\eta}(r)$. Note that the existence of such $\eta$ is guaranteed by the first part of the lemma. 
Now, since $(\Delta_k^n)'$ uniformly converges to $\Delta_k'$ and $\Delta_k'(x)<-C_2$ for all $x\in B_{\eta}(r)$, there exists $n_0(C_2)$ such that we have $(\Delta_k^n)'(x)<-C_2/2$ for all $n\geq n_0(C_2)$ and all $x\in B_{\eta}(r)$. Hence, for  all $n\geq n_0(C_2)$ and $x,\bar x\in B_{\eta}(r)$ with $\bar x\geq  x$ we have
\begin{equation}
    \Delta_k^n\brac{\bar x}-\Delta_k^n\brac{x}=\int_{x}^{\bar x} (\Delta_k^n)'(u)du\leq -\frac{C_2}{2}\brac{\bar x-x}.
    \label{eq:contract}
\end{equation}

Now, define the Lyapunov function $\phi:\mc S^2\to \RR$ as $\phi(x,\bar x)=\log(1+n|\bar x - x|)$. If $G_{X^n,\bar X^n}$ denote the generator of the joint chain, then, due to the independence of the two chains $X^n$ and $\bar X^n$, we have
\begin{align}
    \mc G_{X^n,\bar X^n}\phi(x,\bar x)&=\mc G_{\bar X^n} \phi(x,\bar x)+\mc G_{X^n} \phi(x,\bar x)\nonumber\\
    &\leq\Delta_k^n(\bar x)\partial_{\bar x}\phi(x,\bar x)+\frac{\sigma_k^n(\bar x)}{2n}\sup_{(x,\bar x)\in \mc S^2}{\partial_{\bar x}^2\phi(x,\bar x)}\nonumber\\
    &\hspace{2em}+\Delta_k^n(x)\partial_{x}\phi(x,\bar x)+\frac{\sigma_k^n(x)}{2n}\sup_{(x,\bar x)\in \mc S^2}{\partial_{x}^2\phi(x,\bar x)},\nonumber
\end{align}
where $\partial_z^i$ denotes the $i^{\textrm{th}}$ partial derivative w.r.t $z$.
It is easy to see that for all $\bar x \geq x$ we have
\begin{align}
    &\partial_{\bar x}\phi(x,\bar x)=\frac{n}{1+n(\bar x-x)}=-\partial_x\phi(x,\bar x),\nonumber\\
    &\partial_{\bar x}^2\phi(x,\bar x)=-\frac{n^2}{(1+n(\bar x-x))^2}=\partial_x^2\phi(x,\bar x)<0,\nonumber
\end{align}
Hence, for all $x,\bar x$ satisfying $\bar x-x\geq 1/n$ we have
\begin{align}
\mc G_{X^n\bar X^n}\phi(x,\bar x)\leq \frac{n}{1+n(\bar x-x)}(\Delta_k^n(\bar x)-\Delta_k^n(x))\leq -\frac{nC_2(\bar x -x)}{2(1+n(\bar x-x))}\leq -\frac{C_2}{4}\nonumber,
\end{align}
where in the second inequality we use~\eqref{eq:contract} and in the last inequality we use the fact $\bar x-x\geq 1/n$.

We define $T_{\text{couple}}^n=\inf\{t\geq 0: X^n(t)=\bar X^n(t)\}$,
$T_{\text{exit},\eta}^n=\inf\{t\geq 0: X^n(t)\notin B_{2\eta}(r) \text{ or } \bar X^n(t)\notin B_{2\eta}(r)\}$, and $\tau=T_{\text{couple}}^n\wedge T_{\text{exit},\eta}$. Then applying Dynkin's formula as in the previous part of the lemma we obtain
\begin{equation}
    \EE_{x,\bar x}[\tau]\leq\frac{4\log(1+n(\bar x-x))}{C_2}\leq \frac{4\log(1+2n\eta)}{C_2}\leq C_3\log n\nonumber,
\end{equation}
for all $(x,\bar x)\in B_\eta(r)\times B_\eta(r)$ with $\bar x > x$ for all sufficiently large $n$. Now, using the above and the same line of arguments as in the proof of the previous part it can be shown that $\PP_{x,\bar x}(T_{\text{couple}}^n\geq C'\log n/\epsilon)\leq 2\epsilon$ for any $\epsilon>0$ and all sufficiently large $n$. This completes the proof of the lemma.

\section{Proof of Theorem 4}
\label{proof:stationary}

Since $X^n$ is a birth-death process, for any $x\in [\gamma_1,1-\gamma_0]$ we have
\begin{align*}
    \pi^n\brac{\frac{\floor{nx}}{n}} &= \frac{1}{Z^n} \prod_{j=N_1^n}^{\floor{nx}-1}\frac{q_+^n(j/n)}{q_-^n((j+1)/n)}\\
    &=\frac{1}{Z^n} \prod_{j=N_1^n}^{\floor{nx}-1}\frac{f_k^n(j/n)}{g_k^n((j+1)/n)}\\
    &=\frac{1}{Z^n}\exp\brac{\sum_{j=N_1^n}^{\floor{nx}-1}\log\frac{f_k^n(j/n)}{g_k^n((j+1)/n)}},\\
    &=\frac{1}{Z^n}\exp\brac{n\int_{\gamma_1}^{x}\log(h_k(z))dz+O(\log n)},
\end{align*}
where $Z^n$ is a normalising constant and $h_k=f_k/g_k$. 
Hence, we have
\begin{align}
    \pi^n([\gamma_1,x])=\frac{\sum_{N_1^n\leq a\leq \floor{nx}}\exp\brac{n\int_{\gamma_1}^{a/n}\log(h_k(z))dz+O(\log n)}}{\sum_{N_1^n\leq a\leq n-N_0^n}\exp\brac{n\int_{\gamma_1}^{a/n}\log(h_k(z))dz+O(\log n)}}.\nonumber
\end{align}
From the above, it follows that if $y\mapsto \int_{\gamma_1}^y\log(h_k(z))dz$ is maximised uniquely at
$y^*\in[\gamma_1,1-\gamma_0]$, then $\pi^n([\gamma_1,x])\to 0$ for all $x< y^*$ and $\pi^n([\gamma_1,x])\to 1$ for all $x> y^*$ implying that $\pi^n \Rightarrow \delta_{y^*}$. Similarly, if the maximiser is not unique and the integral $\int_{\gamma_1}^y\log(h_k(z))dz$ is maximised at $m$ discrete values of $y$ given by $\{y_1^*,y_2^*,\ldots,y_m^*\}$, then we have $\pi^n\Rightarrow \frac{1}{m}\sum_{j=1}^m \delta_{y_j^*}$. 
Hence, to establish the theorem it is sufficient to prove the following statements.
\begin{enumerate}
    \item For $\gamma_0=\gamma_1$, the integral $\int_{\gamma_1}^y\log(h_k(z))dz$ is maximised both at $y=r_1$ and $y=r_2$ where $r_1$ and $r_2$ are the minimum and the maximum root of $\Delta_k$ in the interval $[\gamma_1,1-\gamma_0]$.
    \item For $\gamma_0<\gamma_1$, the integral $\int_{\gamma_1}^y\log(h_k(z))dz$ is maximised at $y=r_2$.
    \item For $\gamma_0>\gamma_1$, the integral $\int_{\gamma_1}^y\log(h_k(z))dz$ is maximised at $y=r_1$.
\end{enumerate}
Below we prove each of the above statements in the order they appear above.

When $\gamma_0=\gamma_1$, we have $h_k(z)=1/h_k(1-z)$ and Lemma 1 implies $r_2=1-r_1$.
Hence, $\int_{r_1}^{r_2} \log(h_k(z))dz=\int_{r_1}^{1/2} \log(h_k(z))dz+\int_{1/2}^{1-r_1} \log(h_k(z))dz=\int_{r_1}^{1/2} \log(h_k(z))dz-\int_{1/2}^{r_1} \log(h_k(1-u))du=0$. 
Furthermore, by Lemma 1 we have $\Delta_k(z) >0$ when $z <r_1$ and $\Delta_k(z) <0$ for $z>r_2$. But since  $\log h_k(z)$ and $\Delta_k(z)$ have the same sign for all $z\neq r_1,r_2$, it follows that
the integral $\int_{\gamma_1}^y\log(h_k(z))dz$ is maximised at both $y=r_1$ and $y=r_2$. This proves the second statement.

When $\gamma_0<\gamma_1$, we show that $\int_{r_1}^{r_2} \log(h_k(z))dz >0$. This will then establish the third statement following similar line of arguments as used in the proof of the previous statement. We define $I(\gamma_0,\gamma_1)=\int_{r_1}^{r_2} \log(h_k(z))dz$. We study the function $I(\gamma_0,\gamma_1)$ by varying $\gamma_1$ while keeping $\gamma_0$ constant. We note that when $\gamma_1=\gamma_0$, we have already shown that $I(\gamma_0,\gamma_0)=0$. Now, from the definition of $h_k$, we have
\begin{align*}
        \frac{\partial}{\partial{\gamma_1}}I(\gamma_0, \gamma_1)&=-\int_{r_1}^{r_2}\frac{\partial}{\partial{\gamma_1}} \log((y-\gamma_1)I_{1-y}(k+1,k))dy\nonumber\\
        &\hspace{2cm}-\log(h_k(r_2))\frac{\partial r_2}{\partial{\gamma_1}}+\log(h_k(r_1))\frac{\partial r_1}{\partial{\gamma_1}}\\
        &=\int_{r_1}^{r_2}\frac{1}{y-\gamma_1}dy\\
        & = \log\left( {y-\gamma_1}  \right)\bigg ]_{r_1}^{r_2}\\
        &=\log \left( \frac{r_2-\gamma_1}{r_1-\gamma_1} \right) >0,
\end{align*}
where the second line follows since $\log h_k(r_1)=\log h_k(r_2)=0$. Hence, $I(\gamma_0,\gamma_1)$ is strictly increasing with respect to $\gamma_1$. Combined with the fact that $I(\gamma_0,\gamma_0)=0$, it shows that $I(\gamma_0,\gamma_1)>0$ for $\gamma_1>\gamma_0$. The last statement follows using similar line of arguments as above by establishing $\int_{r_1}^{r_2}\log(h_k(z))dz <0$ when $\gamma_0>\gamma_1$. \qed

\section{Proof of Theorem 5}
\label{sec:subcriticality}

By the first statement of Lemma~\ref{lem:Meta Lemma} and the definitions of $r_1, r_2$ it follows that 
both $r_1$ and $r_2$ are simple and stable roots of $\Delta_k$. It also follows from Lemma~\ref{lem:Meta Lemma} that for all sufficiently small $\xi>0$ the  intervals $[\gamma_1,r_m-\xi]$ and $[r_m+\xi, 1-\gamma_0]$ are regions of attraction of $r_1$ and $r_2$, respectively. 
Now, for any constant $p\neq r_m$ there exists $\xi>0$ such that $p$ lies either in $[\gamma_1,r_m-\xi]$ which corresponds to the region of attraction of $r_1$ or it lies in $[r_m+\xi, 1-\gamma_0]$ which corresponds to the region of attraction of $r_2$. Therefore, by the third statement of the Lemma~\ref{lem:log_hitting_time} it follows that  $\PP_p(T_{\text{enter}}^n(B_{\delta}({r^*}))>C\log n/\epsilon)\leq \epsilon$ for some positive constant $C$ and for all sufficiently large $n$  where $r^*=r_1\indic{p<r_m}+r_2\indic{p>r_m}$. Now since $r_i$ is a stable root for each $i\in \{1,2\}$, by the second statement of Lemma~\ref{lem:log_hitting_time} it follows that $\PP_x(T_{\text{exit}}^n(B_\delta(r_i))\leq \exp(\kappa_\delta n))=o(1)$ for any $x\in B_{\delta/2}(r_i)$. This shows that $T_{\text{exit}}^n(B_\delta(r_i))=\Omega(\exp(\Theta(n)))$ w.h.p. for any starting state $X^n(0)\in B_{\delta/2}(r_i)$.

To establish the mixing time bound, we observe that Theorem~\ref{thm:stationary} implies that for some $i\in \{1,2\}$ we have $\pi^n(B_\delta(r_i))\geq 1/3$ for all sufficiently large $n$ and for all sufficiently small $\delta >0$. But from the definition of maximum total variation distance it follows that $d^n(t)\geq \pi^n(B_\delta(r_i))-\PP_x(X^n(t)\in B_\delta(r_i))\geq 1/3-\PP_x(T^n_{\text{enter}}(B_\delta(r_i))\leq t)$. Now, if $r_{-i}$ is the only element of the set $\{r_1,r_2\}-\{r_i\}$ then for $x \in B_{\delta/2}(r_{-i})$ and sufficiently small $\delta$ from the second statement of Lemma~\ref{lem:log_hitting_time} we have $\PP_x(T^n_{\text{enter}}(B_\delta(r_i))\leq \exp(\kappa_\delta n))\leq \PP_x(T^n_{\text{exit}}(B_{\delta}(r_{-i}))\leq \exp(\kappa_\delta n))=o(1)$. Hence, $d^n(\exp(\kappa_\delta n))\geq 1/3-o(1)\geq 1/4$ which implies that $t_{\text{mix}}^n=\Omega(\exp(\Theta(n)))$. This completes the proof of the theorem.\qed

\section{Proof of Theorem 6}
\label{sec:supercriticality}

From the first statement of Lemma~\ref{lem:Meta Lemma} it follows that  $r$ is a stable root and the whole interval $\mc S$ is the region of attraction of $r$. Hence, from the third statement of Lemma \ref{lem:log_hitting_time}, it follows that for any $x \in \mc S^n$ we have $\EE_{x}[T^n_{\text{enter}}(B_{\delta}(r))]=\EE_{x}[T^n_{\text{enter}}(B_{\delta}(r))\wedge T^n_{\text{exit}}(\mc S)]  \leq C\log n.$

We now prove the mixing time upper bound. Let $\eta > 0$ be as in the last statement of Lemma \ref{lem:log_hitting_time}, and let $\delta' \in (0,\eta)$. Consider two chains $X^n$ and $\bar X^n$ initialised at $x,\bar x \in \mc S^n$, respectively, and evolving independently until they first meet after which they move together. Let $T_{\text{couple}}^n:=\inf \{t\geq 0: X^n(t)=\bar X^n(t)\}$ be the first meeting time of the chains. Using the coupling lemma for Markov chain mixing~\cite{peres_book}, we have
\begin{equation}
    d^n(t)\leq \max_{x,\bar x}\PP_{x,\bar x}(X^n(t)\neq \bar X^n(t))=\max_{x,\bar x}\PP_{x,\bar x}(T_{\text{couple}}^n>t).\nonumber
\end{equation}
Thus, to establish the lemma it is sufficient to show that for every $\epsilon>0$ there exists a constant $C\equiv C(\epsilon)>0$ such that $\PP_{x,\bar x}(T_{\text{couple}}>C\log n)<\epsilon$ for all starting states $x, \bar x \in \mc S^n$. Similar to the chain $X^n$ we denote the first entry time to and first exit from set $B\subseteq \mc S$ of the chain $\bar X^n$ as $\bar T^n_{\text{enter}}(B)$ and $\bar T^n_{\text{exit}}(B)$, respectively.

Define the event $G(t):=\{X^n(t/2)\in B_{\delta'}(r), \bar X^n(t/2) \in B_{\delta'}(r)\}$. Then the complement $G(t)^c$ satisfies
$G(t)^c\subseteq \{T^n_{\text{enter}}(B_{\delta'}(r))>t/2\}\cup \{\bar T^n_{\text{enter}}(B_{\delta'}(r))>t/2\}\cup \{T_{\text{enter}}^n (B_{\delta'}(r)) \leq t/2, X^n(t/2)\notin B_{\delta'}(r)\}\cup \{\bar T_{\text{enter}}^n (B_{\delta'}(r)) \leq t/2, \bar X^n(t/2)\notin B_{\delta'}(r)\}$. Now, for $t=2C'\log n/\epsilon$ it follows from the third statement of Lemma~\ref{lem:log_hitting_time} that
for any $x,\bar x\in \mc S$ both $\PP_{x}(T^n_{\text{enter}}(B_{\delta'}(r))>t/2)$ and $\PP_{x}(\bar T^n_{\text{enter}}(B_{\delta'}(r))>t/2)$ are bounded above by $\epsilon$. To bound the probability of the event $\{T_{\text{enter}}^n (B_{\delta'}(r)) \leq t/2, X^n(t/2)\notin B_{\delta'}(r)\}$ we observe that
\begin{align}
        &\nonumber\PP_{x}\brac{T_{\text{enter}}^n (B_{\delta'}(r)) \leq t/2, X^n(t/2)\notin B_{\delta'}(r)}\\
        & \nonumber\leq\PP_{x}\brac{T_{\text{enter}}^n (B_{\delta'}(r)) \leq t/2, X^n(t/2)\notin B_{\delta'}(r), T_{\text{enter}}^n (B_{\delta'/2}(r))\leq t/2} +\PP_{x}\brac{T_{\text{enter}}^n (B_{\delta'/2}(r))> t/2}\\
        & \nonumber\leq \underbrace{\sup_{y \in B_{\delta'/2}(r)} \PP_{y}\brac{T^n_{\text{exit}}(B_{\delta'}(r))\leq t/2}}_{(a)}+\underbrace{\PP_{x}\brac{T_{\text{enter}}^{n}(B_{\delta'/2}(r))> t/2}}_{(b)}.
\end{align}
Using the second statement of Lemma~\ref{lem:log_hitting_time}, we can bound term (a) as follows. 
\begin{align*}
\PP_{y}\brac{T^n_{\text{exit}}(B_{\delta'}(r))\leq t/2}) &=  \PP_{y}\brac{T^n_{\text{exit}}(B_{\delta'}(r))\leq C'\log n/\epsilon}\\
& \leq \PP_{y}\brac{T^n_{\text{exit}}(B_{\delta'}(r))\leq \exp\left(\frac{\kappa_\delta n}{4}\right)}\\
& \overset{(a)}{\leq} n\sigma_{\max}\exp\left(-\frac{\kappa_\delta n}{4}\right)+2\exp(-\kappa_\delta n)+\exp\left(-\frac{\kappa_\delta n}{2}\right),
\end{align*}
where (a) follows from \eqref{eq:exit time bound}. 
Since this bound does not depend on $y$, for all sufficiently large $n$ we have
$\sup_{y\in B_\delta'(r)}\PP_{y}\brac{T^n_{\text{exit}}(B_{\delta'}(r))\leq t/2}\leq \epsilon$.
We use the third statement of Lemma~\ref{lem:log_hitting_time} to bound term (b) above by $\epsilon$. Hence, we have $\PP_{x}\brac{T_{\text{enter}}^n (B_{\delta'}(r)) \leq t/2, X^n(t/2)\notin B_{\delta'}(r)}\leq 2\epsilon$ for all sufficiently large $n$ and all starting states $x$. We can similarly show that $\PP_{\bar x}(\bar T_{\text{enter}}^n (B_{\delta'}(r)) \leq t/2, \bar X^n(t/2)\notin B_{\delta'}(r))\leq 2\epsilon$ for all $\bar x\in \mc S$ and all sufficiently large $n$.
Hence, combining the above, we have that for any $\epsilon >0$
\begin{equation}
    \PP_{x,\bar x}(G^c(2C'\log n/\epsilon))\leq 6\epsilon,\nonumber
\end{equation}
for all sufficiently large $n$ and all initial states $x$ and $\bar x$.

Now, we focus on the event $T_{\text{couple}}^n>t$. We have
\begin{align}
    \nonumber \PP_{x,\bar x}(T_{\text{couple}}^n>2C' \log n/\epsilon) &\leq \PP_{x,\bar x}(T_{\text{couple}}^n>2C \log n/\epsilon, G(2C'\log n/\epsilon))+\PP_{x,\bar x}(G^c(2C'\log n/\epsilon))\\
    &\nonumber \leq \PP_{x,\bar x}(T_{\text{couple}}^n>2C' \log n/\epsilon |  G(2C'\log n/\epsilon))  +6\epsilon\\
    &\nonumber \overset{(a)} \leq \sup_{y,\bar y \in B_{\delta'}(r)}\PP_{y,\bar y}(T_{\text{couple}}^n>C' \log n/\epsilon)  +6\epsilon\\
    &\nonumber = \sup_{y,\bar y \in B_{\delta'}(r)} \PP_{y,\bar y}(X^n(C' \log n /\epsilon) \neq \bar X^n(C' \log n/\epsilon))  +6\epsilon\\
    &\nonumber \overset{(b)} \leq 7\epsilon,
\end{align}
where $(a)$ follows by the Markov property and $(b)$ follows from the last statement of Lemma \ref{lem:log_hitting_time}. This completes the proof of the theorem.\qed

\section{Proof of Theorem~\ref{thm:criticality} (Criticality)}
\label{proof:criticality}
We divide the proof of Theorem \ref{thm:criticality} into two parts, one for each statement of the theorem.

\subsection{Proof of the first statement}
We provide the proof for the case where  $r_1$ has a multiplicity of two and $r_2$ has multiplicity of one which happens when $(\gamma_0,\gamma_1)\in \Gamma_{\text{critical}}\cap\{\gamma_1>\gamma_0\}$. The proof for $(\gamma_0,\gamma_1)\in \Gamma_{\text{critical}}\cap\{\gamma_1<\gamma_0\}$ is identical with the roles of $\gamma_1$ and $\gamma_0$ reversed. Thus, for the case under consideration, it follows from Lemma~\ref{lem:Meta Lemma} that $r_2$ is a stable root with region attraction $[r_2-\delta,1-\gamma_0]$ for all sufficiently small $\delta>0$.  Furthermore, since $r_1$ has a multiplicity of two, we can write
\begin{align}
    \Delta_k(x)=(x-r_1)^2\tilde{\Delta}_k(x)\nonumber,
\end{align}
where $\tilde{\Delta}_k$ is a polynomial with only one real root of multiplicity one at $r_2\in[\gamma_1,1-\gamma_0]$. Furthermore, by the first statement of Lemma~\ref{lem:Meta Lemma} it follows that $\Delta_k$ is non-negative in $[\gamma_1,r_2)$. Hence, $\Tilde{\Delta}_k(x)>0$ for $x\in[\gamma_1,r_2)$ . Hence, for any $\delta\in (0,r_2-r_1)$, there exists $\Delta_{\max,\delta},\Delta_{\min,\delta}>0$ such that $\Delta_{\max,\delta}>\Tilde{\Delta}_k(x)>\Delta_{\min,\delta}$ or equivalently $\Delta_{\max,\delta}(x-{r_1})^2 \geq \Delta_k(x)\geq\Delta_{\min,\delta}(x-{r_1})^2$ for all $x\in [\gamma_1,r_1+\delta]$. Furthermore, we have $|\Delta_k^n(x)-\Delta_k(x)|=O(1/n)$ for all $x\in [\gamma_1,1-\gamma_0]$.
Hence, there exists $C>0$ such that
\begin{equation}
\Delta_{\max,\delta}(x-{r_1})^2+\frac{C}{n}\geq \Delta_k^n(x)\geq \Delta_{\min,\delta}(x-r_1)^2-\frac{C}{n},
\label{eq:delta_bound}
\end{equation}
for all sufficiently large $n$ and all $x\in [\gamma_1,r_1+\delta]$.

Similarly, since $\sigma_k(x)=f_k(x)+g_k(x)>0$ for all $x\in [\gamma_1,1-\gamma_0]$, there exists $\sigma_{\min},\sigma_{\max}>0$ such that $\sigma_{\min}<\sigma_k(x)<\sigma_{\max}$ for all $x\in [\gamma_1,1-\gamma_0]$. Since $\sigma_k^n$ converges uniformly to $\sigma_k$  we have $$\sigma_{\min}<\sigma_k^n(x)<\sigma_{\max},$$ for all $x\in [\gamma_1,1-\gamma_0]$ and all sufficiently large $n$. In the following subsections, we prove the upper and lower bounds in the first statement separately.

\subsubsection{Upper Bound}

Now, we fix $\delta\in [0,r_2-r_1)$ and define $\tau_{\delta}=\inf\{t\geq 0: X^n(t)\geq r_1+\delta\}$. 
Hence, applying the generator $\mc G_{X^n}$ to $\phi\in C^3[\gamma_1,1-\gamma_0]$ we obtain that for each $x\in [\gamma_1,1-\gamma_0]$
\begin{align}
    \mc G_{X^n} \phi(x)\leq\Delta_k^n(x) \phi'(x)+\frac{\sigma_k^n(x)}{2n} \phi''(x)+\frac{\sigma_k^n(x)}{6n^2}\phi'''_{\max}\nonumber,
\end{align}
where $\phi'''_{\max}=\sup_{x\in \mc S}|\phi'''(x)|$.
Now let Let us define $Z^n(t)=n^{1/3}(X^n(t)-r_1)$ and similarly define $z=n^{1/3}(x-r_1)$. Let $\psi(z)=\phi(x)$. Then we have from~\eqref{eq:generator expansion}
\begin{align}
    \mc G_{X^n} \psi(z)&\nonumber\leq\Delta_k^n(x)n^{1/3}\psi'(z)+\frac{\sigma_k^n(x)}{2n} n^{2/3}\psi''(x)+\frac{\sigma_k^n(x)}{6n^2}n\psi'''_{\max}.\nonumber
\end{align}
The following lemma in crucial to proving the upper bound. 
\begin{lemma}
\label{lem:stein_symmetric}
For each $\kappa_1,\kappa_2,\alpha, n>0$, there exists a non-negative function $\psi\in C^3[-n^{1/3}\kappa_1,n^{1/3}\kappa_2]$ such that the following properties hold:
\begin{enumerate}
    \item The function $\psi$ solves the ODE (Stein's equation)
    \begin{equation}
        \psi''(z)+\alpha z^2 \psi'(z)=-1,
        \label{eq:stein_ode}
    \end{equation}
    with boundary conditions $\psi'(-n^{1/3}\kappa_1)=\psi(n^{1/3}\kappa_2)=0$.
     \item We have $\psi'(z)\leq 0$ for all $z\in [-n^{1/3}\kappa_1, n^{1/3}\kappa_2]$ and $\sup_{z\in [-n^{1/3}\kappa_1,n^{1/3}\kappa_2]}|\psi'(z)|<C_6$ for some positive constant $C_6$ independent of $n$.
    \item We have $\sup_{z\in [-n^{1/3}\kappa_1, n^{1/3}\kappa_2]}\psi(z)<C_7$ for some positive constant $C_7$ independent of $n$.
    \item  For all sufficiently large $n$ we have $\sup_{z\in [-n^{1/3}\kappa_1, n^{1/3}\kappa_2]}|\psi'''(z)|=n^{2/3}\alpha\kappa_1^2$.
\end{enumerate}
\end{lemma}
The proof of this lemma is given in Appendix \ref{sec:stein1}

Choosing $\psi$ according to the above lemma for $\kappa_1=r_1-\gamma_1,\kappa_2=\delta$
and $\alpha=\frac{2\Delta_{\min,\delta}}{\sigma_{\max}}$ we obtain
\begin{align}
    \mc G_{X^n} \psi(z)&\nonumber\leq \Delta_k^n(x)n^{1/3}\psi'(z)+\frac{\sigma_k^n(x)}{2n^{1/3}} \psi''(z)+\frac{\sigma_k^n(x)}{6n}n^{2/3}\alpha \kappa_1^2\\
    & \nonumber\leq \left(\Delta_{\min,\delta}(x-r_1)^2-\frac{C}{n}\right)n^{1/3}\psi'(z)+\frac{\sigma_k^n(x)}{2n^{1/3}} \psi''(z)+\frac{\alpha \kappa_1^2 \sigma_k^n(x)}{6n^{1/3}}\\
    & \nonumber = \frac{\Delta_{\min,\delta}}{n^{1/3}}z^2\psi'(z)+\frac{\sigma_k^n(x)}{2n^{1/3}} \psi''(z)+\frac{\alpha \kappa_1^2 \sigma_k^n(x)}{6n^{1/3}} - \frac{C}{n^{2/3}}\psi'(z)\\
    & \nonumber \leq \frac{\sigma_k^n(x)}{2n^{1/3}}\left(\psi''(z) + \frac{2\Delta_{\min,\delta}}{\sigma_{\max}} z^2\psi'(z) + \frac{\alpha \kappa_1^2}{3}\right) - \frac{C}{n^{2/3}}\psi'(z)\\
    & \nonumber \leq \frac{\sigma_k^n(x)}{2n^{1/3}}\left(-1 + \frac{\alpha \kappa_1^2}{3}\right) + \frac{C C_6}{n^{2/3}}.
\end{align}
where the second line follows from~\eqref{eq:delta_bound} since $\psi'\leq 0$,
and in the last line we have used~\eqref{eq:stein_ode}. 
Then, for all $x \in [\gamma_1, r_1+\delta]$, it holds that
\begin{align}
    \alpha\kappa_1^2&=\frac{2\Delta_{\min,\delta}}{\sigma_{\max}}(\gamma_1-r_1)^2\leq \frac{2\tilde\Delta_k(x)}{\sigma_k(x)}(\gamma_1-r_1)^2=\frac{2\Delta_k(x)}{\sigma_k(x)}\frac{(\gamma_1-r_1)^2}{(x-r_1)^2}\leq 2 \frac{(\gamma_1-r_1)^2}{(x-r_1)^2},\nonumber
\end{align}
and in particular for $x=\gamma_1$ the above inequality reduces to $\alpha \kappa_1^2\leq 2$. Therefore, we can write
\begin{align*}
    \mc G_{X^n} \psi(z) &\leq \frac{\sigma_k^n(x)}{2n^{1/3}}\left(-1 + \frac{2}{3}\right) + \frac{C C_6}{n^{2/3}}\\
    & = -\frac{\sigma_k^n(x)}{6n^{1/3}} + \frac{C C_6}{n^{2/3}}\\
    & \leq -\frac{\sigma_{\min}}{6n^{1/3}} \left(1 -\frac{C C_6}{n^{1/3}}\right)\\
    & < 0,
\end{align*}
for sufficiently large $n$. We are now in a position where we can apply Dynkin's formula. Let $\tau_{\delta}=\inf\{t\geq 0: X^n(t)\geq r_1+\delta\}$. Then applying Dynkin's formula, we can write for any $j>0$
\begin{align}
    0\leq \EE_{\gamma_1}[\psi(Z^n(\tau_\delta\wedge j))]&=\EE_{\gamma_1}[\psi(Z^n(0))]+\EE_{\gamma_1}\sbrac{\int_{0}^{\tau_{\delta}\wedge j}\mc{G}^n\psi(Z^n(s))ds}\nonumber\\
    &\leq \EE_{\gamma_1}[\psi(Z^n(0))]-\frac{\sigma_{\min}}{6n^{1/3}}\left[1 - \frac{6 C C_6}{\sigma_{\min} n^{1/3}}\right]\EE_{\gamma_1}[\tau_{\delta}\wedge j].\nonumber
\end{align}
Now using the fact that $\psi(z)<C'$ for all $z$ for some constant $C'>0$ independent of $n$ we obtain for all sufficiently large $n$ that
\begin{align}
    \EE_{\gamma_1}[\tau_{\delta}\wedge j]\leq \left(\frac{6C'}{\sigma_{\min}-\frac{6CC_6}{n^{1/3}}}\right)n^{1/3}. \nonumber
\end{align}
Since the above holds for any $j>0$, using the dominated convergence theorem we have
\begin{equation}
    \EE_{\gamma_1}[\tau_\delta]=\lim_{j \to\infty}\EE_{\gamma_1}[\tau_{\delta}\wedge j]\leq \left(\frac{6C'}{\sigma_{\min}-\frac{6CC_6}{n^{1/3}}}\right)n^{1/3}.\nonumber
\end{equation}
This shows that $\EE_{\gamma_1}[\tau_{\delta}]=O(n^{1/3})$
which implies that $\EE_{x}[T^n_{\text{enter}}(B_\delta(r))]=O(n^{1/3})$ for any $x<r_2$.
Since $r_2$ is a stable root with region attraction $[r_2-\mu,1-\gamma_0]$ for all sufficiently small $\mu>0$ we have from the third statement of  Lemma~\ref{lem:log_hitting_time} that $\EE_x[T^n_{\text{enter}}(B_\mu(r_2))]=O(\log n)$ for all $x > r_2$. Hence, we have shown $\EE_{x}[T^n_{\text{enter}}(B_\delta(r))]=O(n^{1/3})$ for all $x\in \mc S$.

We now proceed to establish the upper bound on the mixing time. To do this, we employ the same technique as in Theorem \ref{thm:fast_mixing}. We define two independent chains $X^n$ and $\bar X^n$, each evolving according to the rates~\eqref{eq:uprate}-\eqref{eq:downrate} and initialised at $x$ and $\bar x$ respectively. Define $T_{\text{couple}}^n=\inf\{t\geq 0: X^n(t)=\bar X^n(t)\}$. To prove the mixing time upper bound of $O(n^{1/3})$ it suffices to show that 
for every $\epsilon>0$ there exists a constant $C\equiv C(\epsilon)>0$ such that $\PP_{x,\bar x}(T_{\text{couple}}^n>C n^{1/3})\leq\epsilon$ for all starting states $x, \bar x \in \mc S^n$.
We first let $\eta>0$ be such that the last statement of Lemma~\ref{lem:log_hitting_time} holds for any $x,\bar x\in B_{\eta}(r_2)$ and then choose $\delta'\in (0,\eta)$.
Now, similar to the proof of Theorem~\ref{thm:fast_mixing}, we define the event $G(t):=\{X^n(t/2)\in B_{\delta'}(r_2), \bar X^n(t/2) \in B_{\delta'}(r_2)\}$.
Now, similar to the proof of  Theorem~\ref{thm:fast_mixing}, it can be shown that
$\PP_{x,\bar x}(G^c(t))\leq 6\epsilon$ for $t=2Cn^{1/3}/\epsilon$ from which it follows that $\PP_{x,\bar x}(T^n_{couple}>2Cn^{1/3}/\epsilon)\leq 7\epsilon$. This completes the proof of the upper bound.

\subsubsection{Lower Bound}

Choose $\phi(x)=(x-r_1+\epsilon_n)^2$, with $\epsilon_n = n^{-1/3}$. Then for $x \in[r_1-\epsilon_n, r_1]$ and all sufficiently large $n$, we have
\begin{align*}
    \mc G_{X^n} \phi(x) &= \Delta_k^n(x)2(x-r_1+\epsilon_n) + \frac{\sigma_k^n(x)}{n}\\
    & \leq 2\Delta_{\max, \delta}(x-r_1)^2(x-r_1+\epsilon_n) + \frac{\sigma_{\max}}{n}+\frac{2C(x-r_1+\epsilon_n)}{n}\\
    & \leq 2 \Delta_{\max, \delta}\epsilon_n (x-r_1)^2 + \frac{\sigma_{\max}}{n}+\frac{2C\epsilon_n}{n}\\
    & \leq \frac{2\Delta_{\max, \delta} + \sigma_{\max}}{n}\left[1 +\frac{\bar C}{n^{1/3}}\right]\\
    & \leq \frac{C'}{n},
\end{align*}
for sufficiently large $n$. Then, for $x < r_1-\epsilon_n$, we have
\begin{align*}
    \mc G_{X^n} \phi(x) = 2\Delta_k^n(x)(x-r_1+\epsilon_n)+\frac{\sigma_k^n(x)}{n} \leq \frac{\sigma_{\max}}{n}.
\end{align*}
Combining both of the bounds above, we have
\begin{align}
    \mc G_{X^n} \phi(x) < \frac{D}{n}, \quad\forall x \leq r_1,\nonumber
\end{align}
where $D>0$ and is independent of $n$ and $k$. Now define $\tau' = \inf \{ t \geq 0 : |X^n(t)-r_1+\epsilon_n| \geq \epsilon_n \}$ and $\tau'_{r_1} = \inf \{ t \geq 0 : X^n(t) \geq r_1\}$. Clearly, for any $x<r_1$ we have $\EE_x[\tau'_{r_1}]\geq \EE_x[\tau']$. Furthermore, starting from $x=r_1-\epsilon_n$ and applying Dynkin's formula gives
\begin{align*}
    \EE_{r_1-\epsilon_n}[\phi(X^n(\tau' \wedge\delta' n^{1/3}))] &= \phi(X^n(0)) + \EE_{r_1-\epsilon_n}\left[ \int_0^{\tau' \wedge \delta' n^{1/3}} \mc G_{X^n} \phi(X^n(s))ds\right]\\
    & \leq \frac{1}{n^2} +\frac{D}{n}\EE_{r_1-\epsilon_n}[\tau' \wedge \delta' n^{1/3}]\\
    & \leq \frac{1}{n^2} + \frac{D\delta'}{n^{2/3}}.
\end{align*}
We also have
\begin{align}
     \EE_{r_1-\epsilon_n}[\phi(X^n(\tau'\wedge \delta' n^{1/3}))]\geq \EE_{r_1-\epsilon_n}\sbrac{\phi(X^n(\tau'))\indic{\tau'\leq \delta' n^{1/3}}}\geq \epsilon_n^2 \PP_{r_1-\epsilon_n}\brac{\tau'\leq \delta' n^{1/3}}.\nonumber
\end{align}
Now let $\tau'_{r_2-\mu} = \inf \{ t\geq 0 : X^n \geq r_2-\mu\}$, with $0 < \mu < \max \{ r_2-r_1, 1-\gamma_0-r_2\}$. Since $\{ \tau'_{r_2-\mu} \leq \delta n^{1/3}\} \subseteq\{\tau' \leq \delta n^{1/3}\}$, given that $X^n(0) = r_1-\epsilon_n$, then we have that
\begin{align}
\PP_{r_1-\epsilon_n}\brac{\tau'_{r_2-\mu} \leq \delta' n^{1/3}} \leq \PP_{r_1-\epsilon_n}\brac{\tau ' \leq \delta' n^{1/3}} \leq D\delta' + n^{-4/3}.\nonumber
\end{align}
Since $d^n(t)=\max_{x\in [0,1]}\sup_{A\subseteq[\gamma_1,1-\gamma_0]}\abs{\PP_x(X^n(t)\in A)-\pi^n(A)}$.
Hence, choosing $x=(r_1-\epsilon_n)$
and $A=A_{\mu}$ where $A_\mu = [r_2-\mu,r_2+\mu]$, we obtain $$d(t)\geq \abs{\PP_{r_1-\epsilon_n}(X^n(t)\in [r_2-\mu,r_2+\mu])-\pi^n([r_2-\mu,r_2+\mu])}.$$ But, 
since $\pi^n\Rightarrow \delta_{r_2}$ as $n\to \infty$, we have for all sufficiently large $n$
that $\pi^n([r_2-\mu,r_2+\mu])\geq \frac{1}{2}$. Also define $\tau'_{r_2-\mu} = \inf \{ t \geq 0 : X^n(t)\geq r_2-\mu\}$.
Hence,
\begin{align*}
    d(t)&\geq \frac{1}{2}-\PP_{r_1-\epsilon_n}(X^n(t)\in [r_2-\mu,r_2+\mu])\\
    &\geq \frac{1}{2}-\PP_{r_1-\epsilon_n}(\tau'_{r_2-\mu}\leq t).
\end{align*}
Choosing $t=\delta' n^{1/3}$ with $\delta'=1/8D$ in the above we obtain for all sufficiently large $n$
\begin{align*}
    d(\delta' n^{1/3})&\geq \frac{1}{2}-\PP_{r_1-\epsilon_n}(\tau'_{r_2-\mu}\leq \delta' n^{1/3})\geq \frac{1}{2}-D\delta'-n^{-4/3}=\frac{1}{2}-\frac{1}{8} -n^{-4/3} \geq \frac{1}{4}
\end{align*}
This implies that $t_{\text{mix}}^n\geq \frac{n^{1/3}}{8D}=\Omega(n^{1/3})$. 

Therefore, we have shown that $t_{mix} = \Omega(n^{1/3})$ and $t_{mix} = O(n^{1/3})$. This completes the proof of the first statement of Theorem~\ref{thm:criticality}.

\subsection{Proof of the second statement}
Now we turn to the proof of the second statement.
For $(\gamma_0, \gamma_1) \in \Gamma_{\text{crit}}^k$ with $\gamma_0 = \gamma_1=\gamma$ the $\Delta_k$ has a single root of multiplicity three at $r=1/2$, and hence
\begin{align}
    \Delta_k(x)=(x-r)^3\tilde{\Delta}_k(x),\nonumber
\end{align}
where $\tilde{\Delta}_k(x)<0$ for all $x\in [\gamma_1,1-\gamma_0]$. Hence, there exist $\tilde \Delta_{\min},\tilde \Delta_{\max}>0$ such that $-\tilde\Delta_{\max}<\tilde{\Delta}_k(x)<-\tilde \Delta_{\min}<0$. This implies that for all sufficiently large $n$, there exists a $C>0$ such that we have $-\tilde \Delta_{\max}(x-r)^3 + C/n>{\Delta_k^n}(x)>-\tilde \Delta_{\min}(x-r)^3 - C/n\geq0$ for $x\leq r$. As before, the generator acting on a test function $\phi \in C^3(\mc S)$ can be written as
\begin{align}
     \mc G_{X^n} \phi(x)\leq\Delta_k^n(x) \phi'(x)+\frac{\sigma_k^n(x)}{2n} \phi''(x)+\frac{\sigma_k^n(x)}{6n^2}\phi'''_{\max}\nonumber,
\end{align}
where where $\phi'''_{\max}=\sup_{x\in \mc S}|\phi'''(x)|$.

\subsubsection{Upper Bound}
We make the following transinformation $Z^n(t) = n^{1/4}(X^n(t) - r)$ and define $\phi(x)=\psi(z)$ with $z=n^{1/4}(x-r)$. Thus, the generator expansion in terms of $\psi$, can be written as
\begin{align}
    \nonumber  \mc G_{X^n} \psi(z)
    & \leq \Delta_k^n(x)  n^{1/4}\psi'(z)+\frac{\sigma_k^n(x)}{2n^{1/2}} \psi''(z) +\frac{\sigma_k^n(x)}{6n^{5/4}} \max_w|\psi'''(w)|.
\end{align}
Now, we choose $\psi$ according to the following lemma.

\begin{lemma}
\label{lem:stein2}
For each $\kappa,\alpha, n>0$, there exists a non-negative function $\psi\in C^3[-n^{1/4}\kappa,0]$ such that the following properties hold:
\begin{enumerate}
    \item The function $\psi$ solves the ODE (Stein's equation)
    \begin{equation}
        \psi''(z)-\alpha z^3 \psi'(z)=-1,
        \label{eq:stein_ode2}
    \end{equation}
    with boundary conditions $\psi'(-n^{1/4}\kappa)=\psi(0)=0$.
     \item We have $\psi'(z)\leq 0$ for all $z\in [-n^{1/4}\kappa, 0]$ and $\sup_{z\in [-n^{1/4}\kappa, 0]}|\psi'(z)|<C_1$ for some positive constant $C_1$ independent of $n$.
    \item We have $\sup_{z\in [-n^{1/4}\kappa, 0]}\psi(z)<C_2$ for some positive constant $C_2$ independent of $n$.
    \item  For all sufficiently large $n$ we have $\sup_{z\in [-n^{1/4}\kappa,0]}|\psi'''(z)|=\alpha\kappa^3n^{3/4}$.
\end{enumerate}
\end{lemma}
The proof of the above lemma is given in Appendix \ref{sec:stein2}

Choosing $\psi$ according to the above lemma for $\kappa=-\gamma$
and $\alpha=\frac{2\Delta_{\min}}{\sigma_{\max}}$ we obtain
\begin{align}
    \nonumber \mc G_{X^n} \psi(z) 
    & \nonumber \leq - \tilde \Delta_{\min}(x-r)^3n^{1/4}\psi'(z) + \frac{\sigma_k^n(x)}{2n^{1/2}} \psi''(z) +\frac{\sigma_k^n(x)}{6n^{5/4}} \max_{w\in[-\kappa n^{1/4},0]}|\psi'''(w)| + \frac{C}{n^{3/4}}|\psi'(z)|\\
    & \leq - \tilde \Delta_{\min}\frac{z^3}{n^{1/2}}\psi'(z) + \frac{\sigma_k^n(x)}{2n^{1/2}} \psi''(z)  +  \frac{\sigma_k^n(x)}{6n^{5/4}}\alpha\kappa^3n^{3/4}+ \frac{CC_1}{n^{3/4}}\nonumber\\
    & = \frac{\sigma_{k}^{n}(x)}{2n^{1/2}}\left[\psi''(z)-\frac{2 \tilde \Delta_{\min}}{\sigma_{\max}}z^3 \psi'(z)+\frac{\alpha\kappa^3}{3}\right] + \frac{CC_1}{n^{3/4}}\nonumber\\
    & = \frac{\sigma_k^n(x)}{2n^{1/2}}\left( \frac{\alpha \kappa^3}{3}-1\right)+ \frac{CC_1}{n^{3/4}}\nonumber. 
\end{align}

Now consider
\begin{align}
    \alpha \kappa^3 = \frac{2 \tilde \Delta_{\min}}{\sigma_{\max}}\cdot (r-\gamma)^3  < \frac{2 \Delta_k(\gamma)}{\sigma_{\max}}\cdot \frac{(r-\gamma)^3}{(r-\gamma)^3} < \frac{2 \sigma_k(\gamma)}{\sigma_{\max}}\leq 2.\nonumber
\end{align}
Hence,
\begin{align*}
    \mc G_{X^n} \psi(z) &\leq \frac{\sigma_k^n(x)}{2n^{1/2}}\left( \frac{2}{3}-1\right)+ \frac{CC_1}{n^{3/4}}\\
    & \leq - \frac{\sigma_{\min}}{6n^{1/2}}+ \frac{CC_1}{n^{3/4}}\\
    & = -\frac{\sigma_{\min}}{6n^{1/2}}\left[ 1 - \frac{6CC_1}{\sigma_{\min}n^{1/4}}\right].
\end{align*}
Now we are in a position to apply Dynkin's formula to bound the time $T_{\text{enter}}^n(r)=\inf\{t\geq 0: (r-X^n(0))(r-X^n(t))\leq 0\}$, i.e. the first time the root $r$ is reached or crossed. Following similar line of arguments as in the first part we obtain
\begin{align}
    \EE_{\gamma}[T_{\text{enter}}^n(r)]  \leq  \left(\frac{6C_2}{\sigma_{\min}-\frac{6CC_1}{n^{1/4}}}\right)n^{1/2}=C'n^{1/2}\nonumber.
\end{align}
for some constant $C'>0$. Using symmetry we can conclude that $\EE_{1-\gamma}[T_{\text{enter}}^n(r)] \leq C'n^{1/2}$ for all sufficiently large $n$. This implies that $\max_{x\in S^n}\EE_{x}[T^n_{\text{enter}}(r)]\leq C'n^{1/2}$ for all sufficiently large $n$

Now to prove the mixing time upper bound using the hitting time upper bound, we use the following lemma.

\begin{lemma}
    Suppose that $N_0^n=N_1^n=\ceil{n\gamma}$ for some $\gamma\in(0,1/2)$. Let $X^n$ and $\bar X^n$ be two chains, each following the rates in~\eqref{eq:uprate}-\eqref{eq:downrate} and started at $x\in \mc S^n$ and $\bar x\in \mc S^n$, respectively. Assume without loss generality that $|x-1/2|\geq |\bar x -1/2|$. Then we have
    \begin{align}
        T_{\text{couple}}^{n}(X^n,\bar X^n) \leq_{st} T_{\text{enter}}^n(r)\nonumber,
    \end{align}
    where $T^n_{\text{couple}}(X^n,\bar X^n)=\inf\{t\geq 0: X^n(t)=\bar X^n(t)\}$ is the first time the two chains meet and
    $T_{\text{enter}}^n(r)=\inf\{t\geq 0: (r-x)(r-X^n(t))\leq 0\}$, as defined before, is the first time the chain $X^n$ crosses $r$.
\end{lemma}
\begin{proof}
    If the two chains $X^n$, $\bar X^n$ are such that $x=1-\bar x$, then we construct a reflection-symmetric coupling, i.e., whenever the lower chain transitions up one state, the upper chain transitions down one state and vice versa. This is possible since we have $N_0^n=N_1^n$ and therefore $f_k^n(x)=g_k^n(1-x)$ for all $x\in \mc S^n$. When $n$ is even, $r\in \mc S^n$ and both chains meet at $r$ and therefore $T_{\text{couple}}^{n}(X^n,\bar X^n)=T_{\text{enter}}^n(r)$. When $n$ is odd, let the two adjacent states to $r$ be $r_L=1/2-1/2n, r_U=1/2+1/2n \in \mc S^n$. It is obvious that if we continue to use the same reflection-symmetric coupling, then the transition of $(r_L, r_U)\to(r_U, r_L)$ means that the chains do not meet. We modify the coupling when the reflection symmetric chain reaches the state $(r_L, r_U)$. Introduce three independent exponential clocks corresponding to the following transitions. First, at rate $g_k^n(r_L)=f_k^n(r_U)$, both chains move away from $r$. At rate $f_k^n(r_L)$ the lower chain jumps from $r_L\to r_U$, i.e. $(r_L, r_U) \to (r_U, r_U)$ and the chains couple at $r_U$. Finally, at rate $g_{k}^n(r_U)$, the upper chains moves from $r_U \to r_L$, i.e. $(r_L,r_U) \to (r_L, r_L)$ and the chains meet at $r_L$. The first exponential clock to ring determines which transition occurs. After the chains couple, they evolve together. Hence, we again have $T_{\text{couple}}^{n}(X^n,\bar X^n)=T_{\text{enter}}^n(r)$.

    If $x\neq 1-\bar x$ is not true, then define the chain $Y^n$ such that $Y^n(0) = 1 - X^n(0)$. Now couple $X^n$ and $Y^n$ with the reflection-symmetric coupling described above. By the choice of $X^n$, the initial state of $\bar X^n$ lies between those of $X^n$ and $Y^n$. Let $\bar X^n$ evolve independently of both $X^n$ and $Y^n$ until it first meets one of them after which the chains that meet move together. Since the chains are one dimensional birth-death processes, $\bar X^n$ cannot leave the interval bounded by $X^n$ and $Y^n$ without first meeting one of them.  Thus, under this constructed coupling, we must have $T^n_{\text{couple}}(X^n,\bar X^n)\leq T^n_{\text{couple}}(X^n,Y^n)=T^n_{\text{enter}}(r)$.
\end{proof}
Hence, using the above lemma and the coupling lemma for Markov chain mixing we obtain
\begin{align*}
    d^n(t) &\leq \max_{x, \bar x} \PP_{x, \bar x}(T^n_{\text{couple}}(X^n,\bar X^n)>t)\\
    &\leq \max_{x} \PP_{x}(T_{\text{enter}}^n(r) > t)\\
    & \leq \max_{x}\frac{\EE_{x}[T_{\text{enter}}^n(r)]}{t}\\
    & \leq \frac{C'n^{1/2}}{t}.
\end{align*}
Now choosing $t = 4C'n^{1/2}$, we have that $d^n(4C'n^{1/2}) \leq 1/4$. Hence, $t_{mix}^n = O(n^{1/2})$. This concludes the proof of the upper bound. 

\subsubsection{Lower bound}

Choose $\phi(x)=(x-r+\epsilon_n)^2$ with $\epsilon_n=n^{-1/4}$. Then for $x\in[r-\epsilon_n,r]$ and all sufficiently large $n$ we have
\begin{align*}
    \mc G_{X^n}\phi(x)&=\Delta_k^n(x) 2(x-r+\epsilon_n)+\frac{\sigma_k^n(x)}{n}\\
    &\overset{(a)}{\leq} -2d(x-r)^3(x-r+\epsilon_n)+\frac{\sigma_{\max}}{n} + \frac{2C(x-r+\epsilon_n)}{n}\\
    &\leq 2d\epsilon_n(r-x)^3+\frac{\sigma_{\max}}{n} + \frac{2C\epsilon_n}{n}\\
    &\overset{(b)}{\leq}\frac{2d+\sigma_{\max}}{n}+\frac{2C}{n^{5/4}}\\
    & = \frac{2d+\sigma_{\max}}{n} \left( 1 + \frac{\bar C}{n^{1/4}} \right)\\
    & \leq \frac{C'}{n},
\end{align*}
for a sufficiently large $n$, where (a) follows since $\Delta_k^n(x)\leq -d(x-r)^3 +C/n, x\geq r-\epsilon_n$, and $\sigma_k^n(x)<\sigma_{\max}$ and (b) follows because $\abs{x-r}\leq n^{-1/4}$.
Similarly, for $x\leq r-\epsilon_n$ we have 
\begin{equation}
    \mc G_{X^n}\phi(x)=2\Delta_k^n(x) (x-r+\epsilon_n)+\frac{\sigma_k^n(x)}{n}\leq \frac{\sigma_{\max}}{n},\nonumber
\end{equation}
where the inequality follows since $\Delta_k^n(x)>0$ and $x-r+\epsilon_n \leq 0$ in the range considered.
Thus, combining the above bounds we have
\begin{equation}
    \mc G_{X^n} \phi(x)\leq \frac{C}{n}, \quad \forall x \leq r,\nonumber
\end{equation}
where $C$ is some positive constant.
Now,  we define $\tau=\inf\{t\geq0: \abs{x-r+\epsilon_n} \geq 1/n^{1/4}\}$ and $\tau_r=\inf\{t\geq 0:x\geq r\}$. Clearly, for any $x< r$ we have $\EE_{x}[\tau_r]\geq \EE_{x}[\tau]$.
Furthermore, starting from $x=r-n^{-1/4}$ and applying Dynkin's formula we obtain the following:
\begin{align*}
    \EE_{r-\epsilon_n}[\phi(X^n(\tau\wedge \delta n^{1/2}))]&=\EE_{r-\epsilon_n}[\phi(X^n(0))]+\EE_{r-\epsilon_n}\sbrac{\int_{0}^{\tau\wedge \delta n^{1/2}} \mc{G}^n \phi(X^n(s))ds}\\
    &\leq\frac{1}{n^2}+\frac{C}{n}\EE_{r-\epsilon_n}[\tau\wedge \delta n^{1/2}]\\
    &\leq \frac{1}{n^2}+\frac{C}{\sqrt n}\delta.
\end{align*}
But we also have
\begin{align*}
     \EE_{r-\epsilon_n}[\phi(X^n(\tau\wedge \delta n^{1/2}))]\geq \EE_{r-\epsilon_n}\sbrac{\phi(X^n(\tau))\indic{\tau\leq \delta n^{1/2}}}\geq \epsilon_n^2 \PP_{r-\epsilon_n}\brac{\tau\leq \delta n^{1/2}}.
\end{align*}
Hence, we have
\begin{align}
\PP_{r-\epsilon_n}\brac{\tau\leq \delta n^{1/2}}\leq C\delta+n^{-3/2} . \nonumber
\end{align}
Now, we recall that $d(t)=\max_{x\in [0,1]}\sup_{A\subseteq[\gamma_1,1-\gamma_0]}\abs{\PP_x(X^n(t)\in A)-\pi^n(A)}$.
Hence, choosing $x=(r-\epsilon_n)$
and $A=[r,1-\gamma_0]$ we obtain $d(t)\geq \abs{\PP_{r-\epsilon_n}(X^n(t)\in [r,1-\gamma_0])-\pi^n([r,1-\gamma_0])}$. But, 
since $\pi^n\Rightarrow \delta_{r}$ as $n\to \infty$, we have for all sufficiently large $n$
that $\pi^n([r,1-\gamma_0])\geq \frac{1}{2}$. Hence,
\begin{align*}
    d(t)&\geq \frac{1}{2}-\PP_{r-\epsilon_n}(X^n(t)\in [r,1-\gamma_0])\\
    &\geq \frac{1}{2}-\PP_{r-\epsilon_n}(\tau\leq t).
\end{align*}
Choosing $t=\delta n^{1/2}$ with $\delta=1/8C$ in the above we obtain for all sufficiently large $n$
\begin{align*}
    d(\delta n^{1/2})&\geq \frac{1}{2}-\PP_{r-\epsilon_n}(\tau\leq \delta n^{1/2})\geq \frac{1}{2}-C\delta-n^{-3/2}=\frac{1}{2}-\frac{1}{8}-n^{-3/2}\geq \frac{1}{4}.
\end{align*}
This implies that $t_{\text{mix}}^n\geq \frac{n^{1/2}}{8C}=\Omega(n^{1/2})$.

To conclude this section, we have shown that $t_{mix}^n = \Omega(n^{1/2})$ and $t_{mix}^n= \mc O(n^{1/2})$. Therefore, it is true that for $\gamma_0=\gamma_1$, we have $t_{mix}^n = \Theta(n^{1/2})$.

This concludes the proof of Theorem \ref{thm:criticality}.

\section{Conclusion and Future Work}
\label{Sec:conclusion}
We studied the $2k$-choices dynamics in the presence of stubborn agents. We have shown that, as the system size increases, the stationary distribution concentrates on the majority state of the opinion with the larger fraction of stubborn followers. We demonstrate that the convergence to the stationary distribution relies completely on the stubborn fractions. When both fractions are \textit{small}, the system exhibits metastability, taking an exponentially long time to reach steady-state. When either the stubborn fractions are both \textit{large} or the difference between them is \textit{large}, the system converges rapidly to its stationary distribution. When the stubborn fractions lie on the phase transition boundary, convergence to the stationary distribution happens in polynomial time.

Future work should aim to study this process on alternative graph families. In our model, the agents are heterogeneous due to stubbornness. Hence, how the stubborn agents are distributed across different neighbourhoods matters, and would render the analysis much more complex. We would expect, however, that similar results will hold when each neighbourhood has a similar distribution of stubborn agents.

\bibliographystyle{unsrt}
\bibliography{opinion}

\newpage
\appendix
\input{appendix}

\end{document}

%% file: preamble.tex
\usepackage{amsmath}
\usepackage{amsthm}
\usepackage{amssymb}

\newcommand{\mc}[1]{\mathcal{#1}}

\newcommand{\brac}[1]{\left(#1\right)}

\newcommand{\sbrac}[1]{\left[#1\right]}

\newcommand{\abs}[1]{\left\lvert #1 \right\rvert}
\newcommand{\floor}[1]{\left\lfloor #1 \right\rfloor}
\newcommand{\ceil}[1]{\left\lceil #1 \right\rceil}

\newcommand{\indic}[1]{\mathbf{1}_{\brac{#1}}}

\theoremstyle{plain}
\newtheorem{theorem}{Theorem}
\newtheorem{lemma}[theorem]{Lemma}
\newtheorem{proposition}[theorem]{Proposition}

\theoremstyle{definition}
\newtheorem{definition}[theorem]{Definition}

\theoremstyle{remark}

\newcommand{\EE}{\mathbb{E}}

\newcommand{\NN}{\mathbb{N}}
\newcommand{\PP}{\mathbb{P}}
\newcommand{\RR}{\mathbb{R}}

\makeatletter
\newsavebox{\@brx}

\newcommand{\llangle}[1][]{%
  \savebox{\@brx}{\(\m@th{#1\langle}\)}%
  \mathopen{%
    \copy\@brx
    \mkern2mu
    \kern-0.9\wd\@brx
    \usebox{\@brx}%
  }%
}

\newcommand{\rrangle}[1][]{%
  \savebox{\@brx}{\(\m@th{#1\rangle}\)}%
  \mathclose{%
    \copy\@brx
    \mkern2mu
    \kern-0.9\wd\@brx
    \usebox{\@brx}%
  }%
}
\makeatother



%% file: appendix.tex
\section{Proof of Lemma 1}
\label{appendix_sec:lemma1}

%
The function $\Delta_k$ can also be written as
\begin{align}
    \Delta_k(x)&= (1-x-\gamma_0) I_x(k+1,k) - (x-\gamma_1)I_{1-x}(k+1,k),\nonumber
\end{align}
where 
\begin{align}
    I_x(\alpha, \beta) = \sum_{r=a}^{a+b-1}\binom{a+b-1}{r} x^r \left( 1-x\right)^{a+b-1-r}\nonumber,
\end{align}
is the regularized incomplete beta function with parameters $y \in [0,1]$ and $\alpha, \beta \in \mathbb{Z}^+$. This can also be written as $I_x(\alpha,\beta) = B(x;\alpha,\beta)/B(\alpha,\beta)$, where
\begin{align}
    B(x;\alpha,\beta) = \int_{0}^xt^{\alpha-1}(1-t)^{\beta-1} dt\nonumber,
\end{align}
with $B(\alpha,\beta) = B(1;\alpha,\beta)$, where $B(x;\alpha,\beta)$ is the incomplete beta function and $B(\alpha,\beta)$ is the complete beta function. The function $\Delta_k$ then has the following derivative properties:
\begin{enumerate}
    \item The first derivative is
        \begin{align}
            \frac{d}{dx}\Delta_k(x) = \binom{2k}{k}x^{k-1}(1-x)^{k-1}\left[ (2k+1)x(1-x) - k(\gamma_0 x + \gamma_1 (1-x))\right] - 1\nonumber.
        \end{align}
    \item The second derivative is
        \begin{align}
            \frac{d^2}{dx^2}\Delta_k(x) = k \binom{2k}{k}x^{k-2}(1-x)^{k-2}P_{\gamma_0,\gamma_1}(x),
            \label{eq:second_derivative}
        \end{align}
        where $P_{\gamma_0,\gamma_1}$ is a cubic polynomial given by
        \begin{align}
            \nonumber P_{\gamma_0,\gamma_1}(x) = (1-2x)&[(2k+1)x(1-x)-(k-1)(\gamma_0x + \gamma_1(1-x))]\\
            &- (\gamma_0-\gamma_1)x(1-x)\nonumber.
        \end{align}
\end{enumerate}
To prove the first statement, we need the derivative of the regularized incomplete beta function $I_x(a,b)$. It is known that 
\begin{align}
    \frac{d}{dx}I_x(a,b) = \frac{x^{a-1}(1-x)^{b-1}}{B(a,b)}\nonumber, 
\end{align}
where $B(a,b)$ is the complete beta function. Setting $a=k+1$ and $b=k$, we get
\begin{align}
    &\frac{d}{dx}I_x(k+1,k) = \frac{x^{k}(1-x)^{k-1}}{B(k+1,k)},\nonumber\\
    &\frac{d}{dx}I_{1-x}(k+1,k) = -\frac{(1-x)^{k}x^{k-1}}{B(k+1,k)}\nonumber.
\end{align}
Also, note that $B(k+1,k) = (k \binom{2k}{k})^{-1}$. Now the derivative of $\Delta_k$ is
\begin{align}
    \frac{d}{dx}\Delta_k(x) &= -I_x(k+1,k) + (1-x-\gamma_0) \frac{d}{dx}I_x(k+1,k)\nonumber\\
    & \qquad \qquad- I_{1-x}(k+1,k) - (x-\gamma_1)\frac{d}{dx}I_{1-x}(k+1,k)\nonumber\\
    & = -(I_x(k+1,k) + I_{1-x}(k+1,k))\nonumber\\
    & \qquad \qquad + \binom{2k}{k}\left[k (1-x-\gamma_0)x^k(1-x)^{k-1} + k(x-\gamma_1)x^{k-1}(1-x)^k)\right]\nonumber\\
    & = -(I_x(k+1,k) + I_{1-x}(k+1,k))\nonumber\\
    & \qquad \qquad + \binom{2k}{k}x^{k-1}(1-x)^{k-1}\left[k (1-x-\gamma_0)x + k(x-\gamma_1)(1-x)\right].\nonumber
\end{align}
It is true that $I_{1-x}(k+1,k) = 1 - I_x(k,k+1)$. Now we also need an expression for 
\begin{align}
    I_x(k+1,k) + I_{1-x}(k+1,k) & = 1 + I_x(k+1,k) -I_x(k,k+1)\nonumber\\
    & = 1 + k\binom{2k}{k}\int_0^x t^{k}(1-t)^{k-1} - t^{k-1}(1-t)^k dt\nonumber\\
    & = 1 + k\binom{2k}{k}\int_0^x t^{k-1}(1-t)^{k-1}(1-2t)dt\nonumber\\
    & = 1 - \binom{2k}{k}x^k(1-x)^k\nonumber.
\end{align}
Plugging this expression into the derivative equation, we get
\begin{align}
    \frac{d}{dx}\Delta_k(x) &= -1 + \binom{2k}{k}x^{k-1}(1-x)^{k-1}\left[k (1-x-\gamma_0)x + k(x-\gamma_1)(1-x) + x(1-x)\right]\nonumber\\
    & = \binom{2k}{k}x^{k-1}(1-x)^{k-1}\left[(2k+1)x(1-x) -k(\gamma_0x + \gamma_1(1-x))\right] - 1\nonumber\\
    & = \binom{2k}{k}x^{k-1}(1-x)^{k-1}M_{\gamma_0, \gamma_1}(x)-1.\nonumber
\end{align}
This establishes the first derivative result. We will need the derivative of the function $M_{\gamma_0, \gamma_1}$:
\begin{align}
    \frac{d}{dx}M_{\gamma_0, \gamma_1}(x) = (2k+1)(1-2x)+k(\gamma_1-\gamma_0)\nonumber.
\end{align}

The second derivative is then as follows
\begin{align}
    \nonumber \frac{d^2}{dx^2}\Delta_k(x) &= \binom{2k}{k}x^{k-2}(1-x)^{k-2}(k-1)(1-2x)M_{\gamma_0, \gamma_1}(x)\\
    & \hspace{5cm} + \binom{2k}{k}x^{k-1}(1-x)^{k-1}\frac{d}{dx}M_{\gamma_0, \gamma_1}(x)\nonumber\\
    & = \binom{2k}{k}x^{k-2}(1-x)^{k-2}\left( (k-1)(1-2x)M_{\gamma_0, \gamma_1}(x) + x(1-x)\frac{d}{dx}M_{\gamma_0, \gamma_1}(x) \right)\nonumber\\
    & \nonumber = \binom{2k}{k}x^{k-2}(1-x)^{k-2}\bigg[(k-1)(1-2x)\bigg( (2k+1)x(1-x) -k(\gamma_0x + \gamma_1(1-x))\bigg)\nonumber\\
    & \hspace{11.5em} + x(1-x)\bigg((2k+1)(1-2x)\bigg) + k(\gamma_1-\gamma_0)x(1-x)\bigg]\nonumber\\
    & \nonumber =  k\binom{2k}{k}x^{k-2}(1-x)^{k-2}\bigg[ (1-2x)((2k+1)x(1-x\nonumber)\\
    & \hspace{5cm}- (\gamma_0x + \gamma_1(1-x))) + k(\gamma_1-\gamma_0)x(1-x)\bigg]\nonumber\\
    & = k\binom{2k}{k}x^{k-2}(1-x)^{k-2}P_{\gamma_0, \gamma_1}(x)\nonumber.
\end{align}
 
    \subsection{Proof of the first statement}
    Plugging in $\gamma_1$ and $1-\gamma_0$ into $\Delta_k$ gives
    \begin{align}
        &\Delta_k(\gamma_1) = (1-\gamma_0-\gamma_1)I_{\gamma_1}(k+1,k) >0\nonumber\\
        &\Delta_k(1-\gamma_0) =-(1-\gamma_0-\gamma_1)I_{\gamma_0}(k+1,k) < 0\nonumber,
    \end{align}
    since $\gamma_0+\gamma_1 < 1$ and $I_x(k,k+1)$ is a strictly positive function for $x>0$. This also shows that $\Delta_k$ must have at least one real root in $[\gamma_1,1-\gamma_0]$. Below we show that the number of roots of $\Delta_k$ in $[\gamma_1,1-\gamma_0]$ is at most three.
    
    Note that to prove that $\Delta_k$ can have at most three real roots in $[\gamma_1, 1-\gamma_0]$, it is sufficient to prove that $\Delta_k''$ has at most one real root in $[\gamma_1, 1-\gamma_0]$. To see this, assume that $\Delta_k$ has at least four real roots in $[\gamma_1, 1-\gamma_0]$. This by Rolle's theorem implies that $f'_{2k}$ has at least three real roots in $[\gamma_1, 1-\gamma_0]$ which again by Rolle's theorem implies that $f''_{2k}$ has at least two real roots in $[\gamma_1, 1-\gamma_0]$ contradicting the fact $f''_{2k}$ has only one real root in $[\gamma_1, 1-\gamma_0]$.
    
    We now proceed to prove that $\Delta_k''$ has at most one real root in $[\gamma_1, 1-\gamma_0]$. Note from~\eqref{eq:second_derivative}  that to prove this result it is sufficient to prove that  the cubic polynomial $P_{\gamma_0,\gamma_1}$ which can be rewritten as
    \begin{align}
        P_{\gamma_0,\gamma_1}(x) &= (1-2x)[(2k+1)x(1-x)-(k-1)(\gamma_0x + \gamma_1(1-x))] - (\gamma_0-\gamma_1)x(1-x)\nonumber\\
        &= 2(2k+1)x^3 + [(2k-1)(\gamma_0-\gamma_1)-3(2k+1)]x^2\nonumber\\
        &  \quad \quad+((2k+1)-k\gamma_0+(3k-2)\gamma_1)x-(k-1)\gamma_1\nonumber,
    \end{align}
    has at most one real root in $[\gamma_1,1-\gamma_0]$.
    Note that the coefficient of $x^3$ is strictly positive, i.e. $2(2k+1) > 0$ for every $k \geq 1$. Hence, $\lim_{x \to -\infty} P_{\gamma_0,\gamma_1}(x) = - \infty$ and $\lim_{x \to \infty}P_{\gamma_0,\gamma_1}(x) = \infty$. Also note that $P_{\gamma_0,\gamma_1}$ can be written as $P_{\gamma_0,\gamma_1}(x)=x(1-x)((1-2x)G_{\gamma_0,\gamma_1}(x)-(\gamma_0-\gamma_1))$ where 
    \begin{equation}
        G_{\gamma_0,\gamma_1}(x)=(2k+1)-(k-1)\brac{\frac{\gamma_0}{1-x}+\frac{\gamma_1}{x}}.\nonumber
    \end{equation}
    For $x \in [\gamma_1, 1-\gamma_0]$ we have 
    \begin{align}
        G_{\gamma_0,\gamma_1}(x)\geq (2k+1)-(k-1)\brac{\frac{\gamma_0}{\gamma_0}+\frac{\gamma_1}{\gamma_1}}=3.\nonumber
    \end{align}
    Consider the case $\gamma_0<1/2$. Evaluating $P_{\gamma_0, \gamma_1}$ at $1-\gamma_0$ gives
    \begin{align}
        P_{\gamma_0,\gamma_1}(1-\gamma_0) &= \gamma_0(1-\gamma_0)(-(1-2\gamma_0)G_{\gamma_0, \gamma_1}(1-\gamma_0)-(\gamma_0 - \gamma_1))\nonumber\\
        & =- \gamma_0(1-\gamma_0)((1-2\gamma_0)G_{\gamma_0, \gamma_1}(1-\gamma_0)+(\gamma_0 - \gamma_1))\nonumber\\
        & \leq -\gamma_0(1-\gamma_0)(3(1-2\gamma_0) + \gamma_0-\gamma_1)\nonumber\\
        & = -\gamma_0(1-\gamma_0)(3 - 4\gamma_0 - (\gamma_0 + \gamma_1))\nonumber\\
        & < -\gamma_0(1-\gamma_0)(2-4\gamma_0)\nonumber\\
        & < 0\nonumber.
    \end{align}
    If $\gamma_1 < 1/2$, then
    \begin{align}
        P_{\gamma_0,\gamma_1}(\gamma_1) &= \gamma_1(1-\gamma_1)((1-2\gamma_1)G_{\gamma_0,\gamma_1}(\gamma_1) - (\gamma_0 - \gamma_1))\nonumber\\
        & \geq \gamma_1(1-\gamma_1)(3(1-2\gamma_1) - (\gamma_0 - \gamma_1))\nonumber\\
        & = \gamma_1(1-\gamma_1)(3 - 4\gamma_1 -(\gamma_0 + \gamma_1))\nonumber\\
        & > \gamma_1(1-\gamma_1)(2 - 4\gamma_1)\nonumber\\
        & > 0\nonumber.
    \end{align}
    Hence, for $\gamma_0,\gamma_1 < 1/2$, there is at least 1 root of the function in the interval $[\gamma_1, 1-\gamma_0]$, via the IVT. The IVT can also be applied to conclude there is at least 1 root in $(-\infty, \gamma_1)$, since $\lim_{x \to -\infty} P_{\gamma_0,\gamma_1}(x) = - \infty$ and $P_{\gamma_0,\gamma_1}(\gamma_1)>0$. The same can be concluded between $(1-\gamma_0, \infty)$, since $\lim_{x \to \infty} P_{\gamma_0,\gamma_1}(x) = \infty$ and $P_{\gamma_0,\gamma_1}(1-\gamma_0)<0$. Since a cubic has at most three roots, we conclude there is exactly one root in $[\gamma_1, 1-\gamma_0]$ when $\gamma_0, \gamma_1<1/2$.
    
    Alternatively, if $\gamma_1 > 1/2$, either $P_{\gamma_0,\gamma_1}(\gamma_1)>0$, $P_{\gamma_0,\gamma_1}(\gamma_1)=0$, or $P_{\gamma_0,\gamma_1}(\gamma_1)<0$. Note, we can show that 
    \begin{align}
        P_{\gamma_0, \gamma_1}(1/2) &= \frac{1}{4}(\gamma_1-\gamma_0)\nonumber\\
        &> 0\nonumber. 
    \end{align}
    If $P_{\gamma_0,\gamma_1}(\gamma_1)>0$, then IVT is used identically as before to conclude the existence of exactly 1 root in $[\gamma_1, 1-\gamma_0]$. If $P_{\gamma_0,\gamma_1}(\gamma_1)=0$, then $\gamma_1 $ is the root and IVT can be applied in $(-\infty, 1/2)$ and $(1-\gamma_0, \infty)$ to show that it is the only root in $[\gamma_1, 1-\gamma_0]$. Finally, if $P_{\gamma_0,\gamma_1}(\gamma_1)<0$, then since $\lim_{x \to -\infty}P_{\gamma_0, \gamma_1}(x)=-\infty$ and $P_{\gamma_0, \gamma_1}(1/2)>0$, then there is at least one root in $(-\infty,0)$ by IVT. Since $P_{\gamma_0, \gamma_1}(1/2) > 0$ and $P_{\gamma_0, \gamma_1}(\gamma_1) < 0$, there is at least one root in $(1/2, \gamma_1)$ via IVT. Finally, since $P_{\gamma_0, \gamma_1}(1-\gamma_0)$ and $\lim_{x \to \infty} P_{\gamma_0, \gamma_1}(x) = \infty$, then there is at least 1 root in $(1-\gamma_0, \infty)$. Hence there cannot be a root in $(\gamma_1, 1-\gamma_0)$. The proof for the case of $\gamma_1 < 1/2$ immediately follows since $P_{\gamma_0,\gamma_1}(x)=-P_{\gamma_1,\gamma_0}(1-x)$. This concludes the proof of the first statement. 
    
    \subsection{Proof of the second statement}
    Define the function $D_{2k}$ as
    \begin{align}
        D_{2k}(h) = \Delta_k(\gamma_1+h) +\Delta_k(1-\gamma_0-h),
    \end{align}
    From the definition of $\Delta_k$ it follows that
    \begin{align}
        D_{2k}(h) = (1-\gamma_0&-\gamma_1-h)I_{\gamma_1+h}(k+1,k) - hI_{1-\gamma_1-h}(k+1,k)\nonumber\\
        & + h I_{1-\gamma_0-h}(k+1,k) - (1-\gamma_0-\gamma_1-h)I_{\gamma_0+h}(k+1,k)\nonumber,
    \end{align}
    which simplifies to
    \begin{align}
        D_{2k}(h) = (1-\gamma_0&-\gamma_1-h)[I_{\gamma_1+h}(k+1,k) - I_{\gamma_0+h}(k+1,k)]\nonumber\\
        &+ h[I_{1-\gamma_0-h}(k+1,k) - I_{1-\gamma_1-h}(k+1,k)]\nonumber.
    \end{align}
    Since $I_{x}(a,b)$ is an increasing function of $x$ (the regularized incomplete beta function is just the CDF of a binomial distribution, and hence, is an increasing function of $x$), we can conclude that $D_{2k}(h)>0$ for all $h\in[0, 1-\gamma_0-\gamma_1]$ if $\gamma_1>\gamma_0$; $D_{2k}(h)<0$ for all $h\in[0, 1-\gamma_0-\gamma_1]$ if $\gamma_1< \gamma_0$; and $D_{2k}(h)=0$ for all $h \in [0,1-\gamma_0-\gamma_1]$ when $\gamma_0=\gamma_1$. Hence, for $\gamma_1>\gamma_0$, choosing $h=r_1-\gamma_1$ gives
    \begin{align}
        0< D_{2k}(r_1-\gamma_1) &= \Delta_k(r_1) + \Delta_k(1+\gamma_1-\gamma_0-r_1) = \Delta_k(1+\gamma_1-\gamma_0-r_1).
        \label{eq:positivity}
    \end{align}
    Now, since $\Delta_k(1-\gamma_0)<0$ and $r_2$ is the maximum root of $\Delta_k$ below $1-\gamma_0$ we must have $\Delta_k(y)< 0$ for all $y\in (r_2,1-\gamma_0]$. Hence, \eqref{eq:positivity} implies that $1+\gamma_1-\gamma_0-r < r_2$. This establishes the result for statement 2 when $\gamma_1>\gamma_0$. The proof follows similarly for $\gamma_1<\gamma_0$. For $\gamma_1=\gamma_2$ since $D_{2k}(h)=0$ for all $h\in [0,1-\gamma_0-\gamma_1]$ it follows that for every root $r \in [\gamma_1,1-\gamma_0]$ of $\Delta_k$ there is a root of $\Delta_k$ at $1-r$ (since $D_{2k}(r-\gamma_1)=\Delta_k(1-r)=0$). Hence, if $r_1$ is the smallest root in this interval then the largest root in the interval must be $1-r_1=r_2$.

\section{Proof of Proposition 3}
\label{appendix_sec:prop3}
    The phase transition boundary is parameterised by the values of $\gamma_0(w)$ and $\gamma_1(w)$ such that 
    \begin{align}
        \nonumber \Delta_k(w)=0 \quad \text{and} \quad \partial_{x}\Delta_k(x)\bigg |_{x=w} = 0.
    \end{align}
    Doing the above results in two simultaneous equations, linear in $\gamma_0$ and $\gamma_1$:
    \begin{align}
        &(1-\gamma_0-w)P_k(w) - (w-\gamma_1)P_k(1-w)=0 \label{eq: drift_equal_zero_parametric}\\
        & -(P_k(w) + P_k(1-w)) + (1-\gamma_0-w)P_k'(w) + (w-\gamma_1)P_k'(1-w)=0\label{eq: derivative_drift_equal_zero_parametric}.
    \end{align}
    Rearranging~\eqref{eq: derivative_drift_equal_zero_parametric} for $\gamma_0$ and plugging it into~\eqref{eq: drift_equal_zero_parametric}, we arrive at
    \begin{align}
        &\nonumber \gamma_1(w)=w - \frac{P_k(w)(P_k(w)+P_{k}(1-w))}{P'_{k}(1-w)P_k(w)+P_k'(w)P_k(1-w)}.
    \end{align}
    Plugging this expression back into ~\eqref{eq: drift_equal_zero_parametric}, we obtain an expression for $\gamma_0$:
    \begin{align}
        &\nonumber \gamma_0(w)= (1-w) - \frac{P_k(1-w)(P_k(w)+P_{k}(1-w))}{P'_{k}(1-w)P_k(w)+P_k'(w)P_k(1-w)}.
    \end{align}
    Hence, we have that $\gamma_1(w) = \gamma_0(1-w)$. It is true that $P_k(1-w) >P_k(w)$ for $w \in (0,1/2)$. Now consider the derivative of
    \begin{align}
        \nonumber \frac{d}{dt}[t(1-t)]^{k-1} = (k-1)[t(1-t)]^{k-2}(1-2t).
    \end{align}
    Hence, $\frac{d}{dt}[t(1-t)]^{k-1}>0$ for $t\in(0, 1/2)$. Therefore, for $t \leq w$, using this increasing property we have 
    \begin{align}
        \nonumber [t(1-t)]^{k-1} \leq [w(1-w)]^{k-1},
    \end{align}
    and therefore
    \begin{align}
        \nonumber t^k(1-t)^{k-1}\leq \frac{t}{w}w^k(1-w)^{k-1}.
    \end{align}
    Integrating both sides w.r.t $t$ gives
    \begin{align}
        \nonumber \int_0^w t^k(1-t)^{k-1}dt = \frac{w}{2}w^{k}(1-w)^{k-1}.
    \end{align}
    Therefore, using the definition of the regularised incomplete beta function and its derivative, we have that 
    \begin{align}
        \nonumber P_k(w) \leq \frac{w}{2}P_k'(w), 
    \end{align}
    for $w \in (0, 1/2)$. Now we can do the following
    \begin{align}
        \nonumber P_k(w)(P_k(w) + P_k(1-w)) &\leq P_k(w)(P_k(1-w) + P_k(1-w))\\
        &\nonumber = 2P_k(w)P_k(1-w)\\
        &\nonumber \leq wP_k(1-w)P_k'(w)\\
        &\nonumber > w(P_k(1-w)P_k'(w)+P_k(w)P_k'(1-w).
    \end{align}
    Hence we can conclude that
    \begin{align}
        \nonumber \gamma_1(w) &> w - \frac{w(P_k(1-w)P_k'(w)+P_k(w)P_k'(1-w)}{P_k(1-w)P_k'(w)+P_k(w)P_k'(1-w)}\\
        &\nonumber = 0,
    \end{align}
    for $w \in (0,1/2)$. Now we consider $w \to 0$. We have that $P_k(1-w)\to 0$, $P_k(w) = O(w^{k+1})$, and $P_k(1-w)P_k'(w)+P_k(w)P_k'(1-w) = O(w^k)$ as $w \to 0$. Plugging these expressions into $\gamma_0(w)$, we have that $\gamma_0(w) \to -\infty$ as $w \to 0$. Evaluating $\gamma_1(1/2)$, we will need
    \begin{align}
        \nonumber P_k\left(\frac{1}{2}\right) = \frac{1}{2}-\frac{\binom{2k}{k}}{2^{2k+1}}, \qquad P_k'\left(\frac{1}{2}\right) = \frac{k\binom{2k}{k}}{2^{2k-1}}.
    \end{align}
    Then we have
    \begin{align}
        \nonumber\gamma_1\left(\frac{1}{2}\right) &= \frac{1}{2} - \frac{P_k(1/2)(P_k(1/2)+P_{k}(1/2))}{P'_{k}(1/2)P_k(1/2)+P_k'(1/2)P_k(1/2)}\\
        &\nonumber = \frac{1}{2} - \frac{P_k(1/2)}{P'_k(1/2)}\\
        &\nonumber = \frac{1}{2} - \frac{\frac{1}{2}-\frac{\binom{2k}{k}}{2^{2k+1}}}{\frac{k\binom{2k}{k}}{2^{2k-1}}}\\
        &\nonumber = \frac{1}{2} - \frac{1}{4k} + \frac{2^{2k}}{4k\binom{2k}{k}}.
    \end{align}
    Therefore, $\gamma_0(1/2)$ also since $\gamma_1(w) = \gamma_0(1-w)$. Let $\gamma_0(1/2) = \gamma_1(1/2) = \gamma^*_k > 0$. Therefore, as $w$ moves from $1/2$ to $0$, the boundary eventually leaves the valid region $\gamma_0, \gamma_1 > 0$. Since $\gamma_1 > 0$, then it must be that it leaves on the $\gamma_0 = 0$ line. Hence, there exists a minimum value of $w$, i.e. $w_{\min}$ such that $\gamma_0(w_{\min})=0$. Then, we immediately have that $w_{\max} = 1-w_{\min}$, where $\gamma_1(w_{\max})=0$. This completes the proof of the proposition. \qed

\section{Proof of Lemma~\ref{lem:stein_symmetric}}
\label{sec:stein1}
    
For notional simplicity we define $\beta=\alpha/3, l=n^{1/3}\kappa_1, L=n^{1/3}\kappa_2$. Then, it is easy to verify that $z\mapsto\psi(z)$ defined by
\begin{equation}
    \psi(z)=\int_{z}^{L}\int_{-l}^v\exp\brac{-\beta(v^3-u^3)}du dv,
    \label{eq:stein_sol}
\end{equation}
solves~\eqref{eq:stein_ode} with the boundary conditions stated in the lemma. Furthermore, it is easy to see that $\psi\in C^3[-l, L]$ and $\psi\geq 0$. This proves the first property of the lemma.

Taking the derivative of~\eqref{eq:stein_sol} w.r.t $z$ on both sides, we obtain
\begin{equation}
    \psi'(z)=-\int_{-l}^{z}\exp\brac{-\beta(z^3-u^3)}du .
    \label{eq:1d}
\end{equation}
This immediately shows that $\psi'\leq 0$.

To prove the third property, it is sufficient to show that $\psi(-l)$ is bounded above by a constant independent of $l$ and $L$ since by the previous property we know that $\psi$ is non-increasing in the interval $[-l,L]$. Now, by definition we have
\begin{align*}
    \psi(-l)=\int_{-l}^{L}\int_{-l}^v\exp\brac{-\beta(v^3-u^3)}du dv\leq \int_{-\infty}^{\infty}\int_{-\infty}^v \exp\brac{-\beta(v^3-u^3)}du dv,
\end{align*}
where the inequality follows since the integrand is non-negative throughout the entire range of integration. We now show that the integral $I=\int_{-\infty}^{\infty}\int_{-\infty}^v \exp\brac{-\beta(v^3-u^3)}du dv$ is finite.
Substituting $v-u=s$ we obtain
\begin{align*}
    I&=\int_{-\infty}^{\infty}\int_{0}^\infty \exp\brac{-\beta(v^3-(v-s)^3)}ds dv\\
     &=\int_{-\infty}^{\infty}\int_{0}^\infty \exp\brac{-\beta(3v^2s-3vs^2+s^3)}ds dv\\
     &=\int_{0}^\infty \int_{-\infty}^{\infty}\exp\brac{-\beta\brac{3s\brac{v-\frac{s}{2}}^2+\frac{s^3}{4}}}dv ds\\
     &=\int_{0}^{\infty} \exp\brac{-\beta\frac{s^3}{4}}\int_{-\infty}^{\infty}\exp\brac{-3\beta s\brac{v-\frac{s}{2}}^2}dv ds\\
     &=\int_{0}^{\infty} \exp\brac{-\beta\frac{s^3}{4}}\sqrt{\frac{\pi}{3\beta s}}ds\\
     &=\frac{2^{1/3}\sqrt{\pi}}{3\sqrt{3}\beta^{2/3}}\int_{0}^\infty t^{-5/6}\exp(-t)dt\\
     &=\frac{2^{1/3}\sqrt{\pi}}{3\sqrt{3}\beta^{2/3}}\Gamma\brac{\frac{1}{6}}<\infty.
\end{align*}
This establishes the third property.

To establish the last property we note that $\psi'''(z)+2\alpha z\psi'(z)+\alpha z^2\psi''(z)=0$ and $\psi''(z)=-1-\alpha z^2\psi'(z)$. Combining the above two, we obtain
\begin{align}
    \psi'''(z)=\alpha z^2+\alpha^2z^4\psi'(z)-2\alpha z\psi'(z)&=\alpha z^2+(\alpha^2z^4-2\alpha z)\psi'(z)\nonumber\\
    &=\alpha z^2+(2\alpha z-\alpha^2z^4)\abs{\psi'(z)}.
    \label{eq:3d}
\end{align}
Thus, to bound $|\psi'''(z)|$ we first need to bound $|\psi'(z)|$.
Choose  a constant  $c>0$ sufficiently large such that $\alpha^2 z^4-2\alpha z\geq 0$ whenever $z\geq c$.
For $z<-c<0$ using~\eqref{eq:1d} we have
\begin{align}
\abs{\psi'(z)}&=\int_{-l}^{z}\exp\brac{-\beta(z^3-u^3)}du\nonumber\\
&\leq \int_{-\infty}^{z}\exp\brac{-\beta(z^3-u^3)}du\nonumber\\
&= \int_{-\infty}^{z}\exp\brac{-\beta(z-u)(z^2+zu+u^2)}du\nonumber\\
&\leq \int_{-\infty}^{z}\exp\brac{-3\beta z^2(z-u)}du\nonumber\\
&=\int_{0}^{\infty}\exp\brac{-3\beta z^2s}ds\nonumber\\
&=\frac{1}{3\beta z^2}=\frac{1}{\alpha z^2},
\label{eq:bound_left}
\end{align}
where the second inequality follows since for $u\leq z\leq 0$ we have $zu+u^2\geq 2z^2$. Similarly, for $z\in [-c,c]$ we have using~\eqref{eq:1d}
\begin{align}
|\psi'(z)|&=\exp(-\beta z^3)\int_{-l}^z\exp(\beta u^3)du\nonumber\\
&\leq \exp(\beta c^3)\int_{-\infty}^{c}\exp\brac{\beta u^3}du \nonumber\\
&\leq \exp(\beta c^3)\brac{\int_{-\infty}^{0}\exp\brac{\beta u^3}du+\int_0^c \exp\brac{\beta u^3}du}\nonumber\\
&\leq \frac{\exp(\beta c^3)}{3\beta^{1/3}}\int_0^{\infty}s^{-2/3}\exp(-s)ds+c\exp(2\beta c^3)\nonumber\\
&= K_1,
\label{eq:bound_middle}
\end{align}
where $K_1=\frac{\exp(\beta c^3)}{3\beta^{1/3}}\Gamma\brac{\frac{1}{3}}+c\exp(2\beta c^3)$.
For $z\geq c$ using~\eqref{eq:1d} we have
\begin{align}
\abs{\psi'(z)}&\leq\int_{-\infty}^{z}\exp\brac{-\beta(z^3-u^3)}du\nonumber\\
&=\exp(-\beta z^3)\int_{-\infty}^{0}\exp\brac{\beta u^3}du+\int_{0}^{z}\exp\brac{-\beta(z^3 - u^3)}du\nonumber\\
& \leq K_2 \exp(-\beta z^3) +\int_{0}^{z}\exp\brac{-\beta(z^3 - u^3)}du\nonumber\\
& \overset{(a)}{=} K_2 \exp(-\beta z^3) +\frac{1}{3\beta z^2}\int_0^{\beta z^3} \left(1-\frac{t}{\beta z^3}\right)^{-2/3}e^{-t}dt\nonumber\\
& \overset{(b)}{=} K_2 \exp(-\beta z^3) +\frac{1}{3\beta z^2}\int_0^{m} \left(1-\frac{t}{m}\right)^{-2/3}e^{-t}dt\nonumber\\
& = K_2 \exp(-\beta z^3) +\frac{1}{3\beta z^2}\left[\int_0^{m/2} \left(1-\frac{t}{m}\right)^{-2/3}e^{-t}dt + \int_{m/2}^{m} \left(1-\frac{t}{m}\right)^{-2/3}e^{-t}dt\right]\nonumber\\
& \overset{(c)}{\leq} K_2 \exp(-\beta z^3) +\frac{1}{3\beta z^2}\left[\int_0^{m/2} \left(1+\frac{3t}{m}\right)e^{-t}dt + \int_{m/2}^{m} \left(1-\frac{t}{m}\right)^{-2/3}e^{-t}dt\right]\nonumber\\
& \leq K_2 \exp(-\beta z^3) +\frac{1}{3\beta z^2}\left[ \left(1+\frac{3}{m}\right) + \int_{m/2}^{m} \left(1-\frac{t}{m}\right)^{-2/3}e^{-t}dt\right]\nonumber\\
& \leq K_2 \exp(-\beta z^3) +\frac{1}{3\beta z^2}\left[ \left(1+\frac{3}{m}\right) + e^{-\frac{m}{2}} \int_{m/2}^{m} \left(1-\frac{t}{m}\right)^{-2/3}dt\right]\nonumber\\
& = K_2 \exp(-\beta z^3) +\frac{1}{3\beta z^2}\left[ \left(1+\frac{3}{m}\right) +  \frac{3m}{2^{1/3}} e^{-\frac{m}{2}} \right]\nonumber\\
& = K_2 \exp(-\beta z^3) + \frac{1}{3\beta z^2} + \frac{1}{\beta^2 z^5} + 2^{-1/3}ze^{-\beta z^3 / 2}\nonumber\\
& = K_2 \exp(-\beta z^3) + \frac{1}{\alpha z^2} + \frac{1}{\beta^2 z^5} + 2^{-1/3}ze^{-\beta z^3 / 2},
\label{eq:bound_right}
\end{align}
where $K_2=\int_{-\infty}^{0}\exp\brac{\beta u^3}du = \frac{1}{3\beta^{1/3}} \Gamma(1/3)$, (a) follows by substituting $\beta(z^3-u^3)=t$,(b) follows from the letting $m=\beta z^3$, and (c) follows from the fact that $(1-x)^{-2/3} \leq 1+3x$ for $0 \leq x \leq 1/2$.
For $z\geq c$ we also obtain the following lower bound using~\eqref{eq:1d}
\begin{align}
\abs{\psi'(z)}&\geq\int_{0}^{z}\exp\brac{-\beta(z^3-u^3)}du\nonumber\\
&\geq \int_{0}^{z}\exp\brac{-3\beta z^2(z-u)}du\nonumber\\
&=\int_{0}^{z}\exp\brac{-3\beta z^2s}ds=\frac{1-\exp(-\alpha z^3)}{\alpha z^2}\nonumber,
\end{align}
where the second inequality follows
since $z^3-u^3=(z-u)(z^2+zu+u^2)\leq 3z^2(z-u)$ for $0\leq u\leq z$.

Equipped with the above bounds on $|\psi'(z)|$ we now bound $|\psi'''(z)|$.

For $-l\leq z< -c$ using~\eqref{eq:3d} we have $\psi'''(z)\leq \alpha z^2\leq n^{2/3}\alpha \kappa_1^2$ since $2\alpha z-\alpha^2 z^4<0$ for $z<-c<0$. For the same region, using~\eqref{eq:3d} and~\eqref{eq:bound_left} we  have $\psi'''(z)\geq 2/z$ implying $\abs{\psi'''(z)}=\max\{\psi'''(z),-\psi'''(z)\}\leq \max\{\kappa_1^2\alpha n^{2/3},2/|z|\}\leq \max\{\kappa_1^2\alpha n^{2/3},2/c\}\leq \kappa_1^2\alpha n^{2/3}$ for all sufficiently large $n$.

From~\eqref{eq:3d} and~\eqref{eq:bound_middle} we have $$\sup_{z\in[-c,c]}|\psi'''(z)| \leq \sup_{z\in[-c,c]}\alpha z^2+(2\alpha |z|+\alpha^2 z^4)|\psi'(z)|\leq \alpha c^2+(2\alpha c+\alpha^2 c^4) K_1\leq n^{2/3}\alpha \kappa_1^2,$$
for all sufficiently large $n$.

Finally, for $z> c$, using~\eqref{eq:3d} and~\eqref{eq:bound_right} we have
\begin{align*}
    \psi'''(z) &\geq \alpha z^2+(2\alpha z - \alpha^2 z^4)\left(K_2 \exp(-\beta z^3) + \frac{1}{\alpha z^2} + \frac{1}{\beta^2 z^5} + 2^{-1/3}ze^{-\beta z^3 / 2}\right)\\
    & = \frac{18}{\alpha z^4}-\frac{7}{z}+\alpha K_2z(2-\alpha z^3)\exp\left(-\frac{\alpha z^3}{3}\right) + 2^{-1/3}\alpha z^2\left(2-\alpha z^3\right)\exp\left(-\frac{\alpha z^3}{6}\right)\\
    &= \frac{18}{\alpha z^4}-\frac{7}{z} + \alpha z(2-\alpha z^3)\left( K_2 \exp\left(-\frac{\alpha z^3}{3}\right) + 2^{-1/3} z \exp\left(-\frac{\alpha z^3}{6}\right)\right)\\
    & \geq -\frac{7}{z} - (\alpha^2 z^4)\left( K_2 \exp\left(-\frac{\alpha z^3}{3}\right) + 2^{-1/3} z \exp\left(-\frac{\alpha z^3}{6}\right)\right),
\end{align*}
and
\begin{align}
\psi'''(z)\leq \alpha z^2+(2\alpha z-\alpha^2z^4)\frac{1-\exp(-\alpha z^3)}{\alpha z^2}\leq\frac{2}{z}+\alpha z^2\exp(-\alpha z^3)\nonumber.
\end{align}
Therefore, for suitably large $c>0$, for all $z\geq c$ and sufficiently large $n$ we have
\begin{align*}
\sup_{z\in [c,\kappa_2n^{1/3}]}|\psi'''(z)|&\leq \sup_{z\in [c,\kappa_2n^{1/3}]}\max \bigg \{\frac{2}{z}+\alpha z^2\exp(-\alpha z^3),\\
&\hspace{3.5cm}\frac{7}{z} + (\alpha^2 z^4)\left( K_2 \exp\left(-\frac{\alpha z^3}{3}\right) + 2^{-1/3} z \exp\left(-\frac{\alpha z^3}{6}\right)\right)\bigg\}\\
&=\max \bigg \{\frac{2}{c}+\alpha c^2\exp(-\alpha c^3),\\
&\hspace{3.5cm}\frac{7}{c} + (\alpha^2 c^4)\left( K_2 \exp\left(-\frac{\alpha c^3}{3}\right) + 2^{-1/3} c \exp\left(-\frac{\alpha c^3}{6}\right)\right)\bigg\}.\\
\end{align*}
This also concludes the second property of the lemma. Since the RHS of the above is a constant, for all sufficiently large $n$ we have $\sup_{z\in [c,\kappa_2n^{1/3}]}|\psi'''(z)|\leq \kappa_1^2\alpha n^{2/3}$.

Hence, combining the bounds for all regions we obtain 
$\sup_{z\in [-l,L]}|\psi'''(z)|\leq n^{2/3}\alpha \kappa_1^2$ for all sufficiently large $n$.
But at $z=-l$ we have $\psi'(-l)=0$. Hence, $\psi'''(-l)=\alpha l^2=n^{2/3}\alpha \kappa_1^2$. Hence, we have shown that the supremum is achieved and this completes the proof of the last statement.\qed

\section{Proof of Lemma~\ref{lem:stein2}}
\label{sec:stein2}

    It is easy to verify that $\psi$ as defined below satisfies~\eqref{eq:stein_ode2} with the stated boundary conditions.
\begin{align}
    \psi(z)=\int_z^0\int_{-\kappa n^{1/4}}^v\exp\brac{\frac{\alpha}{4}(v^4-u^4)}du dv\nonumber.
\end{align}
Clearly, we have
\begin{align}
    \psi'(z)=-\int_{-\kappa n^{1/4}}^z\exp\brac{\frac{\alpha}{4}(z^4-u^4)}du \leq 0\nonumber.
\end{align}
This verifies the second property.

The second property also implies that $\psi(z)$ takes its maximum value in the range $[-\kappa n^{1/4}, 0]$ when $z=-\kappa n^{1/4}$. Hence, to prove the third property it is sufficient to show that $\sup_{n}\psi(-\kappa n^{1/4})< \infty$. Now, we have
\begin{align*}
    \psi(-\kappa n^{1/4})&=\int_{-\kappa n^{1/4}}^0\int_{-\kappa n^{1/4}}^v\exp\brac{\frac{\alpha}{4}(v^4-u^4)}du dv\\
    &=\int_{-\kappa n^{1/4}}^0 \int_{u}^0 \exp\brac{\frac{\alpha}{4}(v^4-u^4)} dv du\\
    &\overset{(a)}{\leq} \int_{-\infty}^0 \int_{u}^0 \exp\brac{\frac{\alpha}{4}(v^4-u^4)} dv du\\
    &=\int_{-\infty}^0 \int_{u}^0 \exp\brac{\frac{\alpha}{4}(v-u)(v+u)(v^2+u^2)} dv du\\
    &\overset{(b)}{\leq} \int_{-\infty}^0 \int_{u}^0 \exp\brac{\frac{\alpha}{4}(v-u)u^3} dv du\\
    &=\int_{-\infty}^0 \int_{0}^{-u} \exp\brac{\frac{\alpha}{4}s u^3} ds  du\\
    &=\int_{-\infty}^0 \frac{4}{\alpha u^3}\brac{\exp\brac{-\frac{\alpha}{4} u^4}-1}du\\
    &=\int_{0}^{\infty} \frac{4}{\alpha u^3}\brac{1-\exp\brac{-\frac{\alpha}{4} u^4}}du\\
    &\leq \int_{0}^{1} \frac{4}{\alpha u^3}\brac{1-\exp\brac{-\frac{\alpha}{4} u^4}}du+ \int_{1}^{\infty} \frac{4}{\alpha u^3}du\\
    &\overset{(c)}{\leq} \int_{0}^{1} \frac{4}{\alpha u^3}\frac{\alpha}{4}u^4du+\frac{2}{\alpha}\\
    &=\frac{1}{2}+\frac{2}{\alpha},
\end{align*}
where (a) follows since the integrand is non-negative throughout the entire range of integration, (b) follows since $u+v\leq u\leq 0$ and $u^2+v^2\geq u^2\geq 0$ and $v-u\geq 0$, and (c) follows since $e^{-x}\geq 1-x$ for all $x\geq 0$. This establishes the third property of the lemma. 

To establish the final property of the lemma, note that $\psi'''(z) -3\alpha z^2\psi'(z)-\alpha z^3 \psi''(z)=0$ and $\psi''(z) = \alpha z^3 \psi'(z)-1$. Therefore, 

\begin{align}
    \psi'''(z) = -\alpha z^3 + (\alpha^2 z^6 + 3\alpha z^2) \psi'(z) = -\alpha z^3 - (\alpha^2 z^6 + 3\alpha z^2) |\psi'(z)|.\nonumber
\end{align}
It is clear that $\alpha^2 z^6 + 3\alpha z^2\geq0$ for any $z \in [-\kappa n^{1/4}, 0]$. Consider $|\psi'(z)|$ in the region $z \in [-\kappa n^{1/4}, -\epsilon]$. Then we have that
\begin{align*}
    |\psi'(z)| &= \int_{-\kappa n^{1/4}}^{z} \exp \left(\frac{\alpha}{4} (z^4-u^4)\right)du\\
    & \leq \int_{-\infty}^{z} \exp \left(\frac{\alpha}{4} (z^4-u^4)\right)du\\
    & = \int_{-\infty}^{z} \exp \left(\frac{\alpha}{4} (z^2-u^2)(z^2+u^2)\right)du\\
    & \leq \int_{-\infty}^{z} \exp \left(\frac{\alpha}{4} (z^2-u^2)(2z^2)\right)du\\
    & = \int_{-\infty}^{z} \exp \left(\frac{\alpha}{2} z^2(z-u)(z+u)\right)du\\
    & \leq \int_{-\infty}^{z} \exp \left(\frac{\alpha}{2} z^2(z-u)(2z)\right)du\\
    & = \int_{-\infty}^{z} \exp \left(\alpha z^3(z-u)\right)du\\
    & = \int_0^{\infty}\exp(\alpha z^3 s )ds\\
    & = - \frac{1}{\alpha z^3}.
\end{align*}
Therefore, for $z \in [-\kappa n^{1/4}, -\epsilon]$, we have 
\begin{align}
    \psi'''(z) \geq -\alpha z^3 - (\alpha^2 z^6 + 3\alpha z^2)\left(\frac{-1}{\alpha z^3}\right) = \frac{3}{z}.\nonumber
\end{align}
We also have
\begin{align}
    \psi'''(z) \leq - \alpha z^3.\nonumber
\end{align}
Therefore $3/z \leq \psi'''(z) \leq -\alpha z^3$. Hence, $|\psi'''(z)| \leq \max_{z \in [-\kappa n^{1/4}, -\epsilon]}\{\alpha |z|^3, 3/|z|\}$. For $n$, sufficiently large, we can conclude that $\sup_{z \in [-\kappa n^{1/4}, -\epsilon]}|\psi'''(z)| \leq \alpha \kappa^3 n^{3/4}$.

Now we focus on the region $z \in (-\epsilon,0]$. In this region we have that

\begin{align*}
    |\psi'(z)| &= \int_{-\kappa n^{1/4}}^{z} \exp \left(\frac{\alpha}{4} (z^4-u^4)\right)du\\
    & = \exp\left(\frac{\alpha}{4}z^4\right) \int_{-\kappa n^{1/4}}^{z}\exp\left(-\frac{\alpha}{4}u^4\right)du\\
    & \leq \exp\left(\frac{\alpha}{4}\epsilon^4\right) \int_{-\kappa n^{1/4}}^{z}\exp\left(-\frac{\alpha}{4}u^4\right)du\\
    & \leq \exp\left(\frac{\alpha}{4}\epsilon^4\right) \int_{-\infty}^{0}\exp\left(-\frac{\alpha}{4}u^4\right)du\\
    & = \frac{\Gamma(1/4) \exp\left(\frac{\alpha}{4}\epsilon^4\right)}{2\sqrt{2} \alpha^{1/4}}\\
    & = C_{\alpha, \epsilon}.
\end{align*}
Then, for $z \in (-\epsilon,0]$, we can write $\psi'''(z) \leq -\alpha z^3$ and 
\begin{align*}
    \psi'''(z) &\geq -\alpha z^3 - (\alpha^2 z^6 + 3\alpha z^2)C_{\alpha, \epsilon}\\
    & \geq -(\alpha^2 \epsilon^6+3\alpha \epsilon^2)C_{\alpha, \epsilon}\\
    & = -(\alpha^2 \epsilon^6+3\alpha \epsilon^2)\frac{\Gamma(1/4) \exp\left(\frac{\alpha}{4}\epsilon^4\right)}{2\sqrt{2} \alpha^{1/4}}.
\end{align*}
Bringing these 2 bounds together, we can conclude that for every $z \in (-\epsilon, 0]$ and sufficiently large $n$, we have that 
\begin{align}
    |\psi'''(z)| \leq \max \left \{\alpha\epsilon^3,  (\alpha^2 \epsilon^6+3\alpha \epsilon^2)\frac{\Gamma(1/4) \exp\left(\frac{\alpha}{4}\epsilon^4\right)}{2\sqrt{2} \alpha^{1/4}}\right\}.\nonumber
\end{align}
Then bringing together both regions we have that for all $z \in [-\kappa n^{1/4}, 0]$:
\begin{align*}
    |\psi'''(z)| \leq \max \left \{ \alpha \kappa^3 n^{3/4}, \frac{3}{\epsilon}, \alpha\epsilon^3,  (\alpha^2 \epsilon^6+3\alpha \epsilon^2)\frac{\Gamma(1/4) \exp\left(\frac{\alpha}{4}\epsilon^4\right)}{2\sqrt{2} \alpha^{1/4}}\right \},
\end{align*}
which for fixed $\alpha, \kappa, \epsilon > 0$ and sufficiently large $n$,  allows us to conclude that $|\psi'''(z)| \leq \alpha \kappa^3 n^{3/4}$. This concludes the final statement of the lemma. \qed